\documentclass[reqno, 12pt]{amsart}
\usepackage{amsmath}
\usepackage{amsthm}
\usepackage{amsfonts}
\usepackage{amssymb}
\usepackage{mathrsfs}

\newcommand{\mc}{\mathcal}

\renewcommand{\Re}{\mathrm{Re}\,}
\renewcommand{\Im}{\mathrm{Im}\,}

\newcommand{\N}{\mathbb{N}}
\newcommand{\R}{\mathbb{R}}
\newcommand{\C}{\mathbb{C}}
\newcommand{\Z}{\mathbb{Z}}

\renewcommand{\S}{\mathbb{S}}

\newcommand{\I}{\,\mathrm{i}\,}

\newcommand{\sgn}{\mathrm{sgn}\,}

\newtheorem{lemma}{Lemma}[section]
\newtheorem{proposition}[lemma]{Proposition}
\newtheorem{theorem}[lemma]{Theorem}
\newtheorem{corollary}[lemma]{Corollary}
\theoremstyle{remark}
\newtheorem{remark}[lemma]{Remark}

\theoremstyle{definition}
\newtheorem{definition}[lemma]{Definition}

\numberwithin{equation}{section}

\title[Spectral mapping]{A spectral mapping theorem for perturbed Ornstein-Uhlenbeck operators
on $L^2(\R^d)$}

\author{Roland Donninger}
\address{Rheinische Friedrich-Wilhelms-Universit\"at Bonn,
Mathematisches Institut, Endenicher Allee 60, D-53115 Bonn, Germany}
\email{donninge@math.uni-bonn.de}

\author{Birgit Sch\"orkhuber}
\address{Fakult\"at f\"ur Mathematik,
Universit\"at Wien, Oskar-Morgenstern-Platz 1, A-1090 Vienna, Austria}
\email{birgit.schoerkhuber@univie.ac.at}

\thanks{Roland Donninger is supported by the Alexander von Humboldt Foundation via
a Sofja Kovalevskaja Award endowed by the German Federal Ministry of Education
and Research. Birgit Sch\"orkhuber is supported by the Austrian Science Fund
(FWF) via the Hertha Firnberg Program, Project Nr. T739-N25.}

\begin{document}
\begin{abstract}
We consider Ornstein-Uhlenbeck operators on $L^2(\R^d)$ perturbed by a radial potential $V$.
Under weak assumptions on $V$ we prove a spectral mapping theorem for the 
generated semigroup. The proof relies on a perturbative construction of the resolvent,
based on angular separation, and the Gearhart-Pr\"u\ss{} Theorem.
\end{abstract}

\maketitle

\section{Introduction}
\noindent We consider a class of operators generated by the formal
differential expression
\begin{equation}
\label{eq:L}
 \mc L_V u(x):=\Delta u(x)-2 x\cdot \nabla u(x)+V(|x|)u(x) 
 \end{equation}
for $u:\R^d\to\C$ with a complex-valued radial 
potential $V: [0,\infty)\to \C$.
The study of elliptic and parabolic problems with unbounded coefficients is motivated by
many applications in science, engineering, and economics. 
Operators as in \eqref{eq:L} are prototypes of this kind and attracted
a lot of interest in the mathematical literature. We refer to the monograph \cite{BerLor07} for recent developments 
in this field.

A natural space for the analysis of $\mc L_V$ is the weighted $L^2_w(\R^d)$ with the Gau\ss ian
weight $w(x)=e^{-|x|^2}$. The reason for this is that the \emph{free operator}
$\mc L_0$ is symmetric on $L^2_w(\R^d)$.
However, in this paper we consider $\mc L_V$ on $L^2(\R^d)$ without weight,
which is motivated in the following.

With a suitable domain (see below) the formal expression $\mc L_V$ has a realization
as an unbounded operator $L$ on $L^2(\R^d)$
which generates a strongly continuous one-parameter semigroup $\{S(t): t\geq 0\}$.
This shows that the $L^2$-setting without weight is very natural, too.
The operator $L$ is highly non-self-adjoint and the complex half-plane
$\{z\in \C: \Re z\leq d\}$ is contained in its spectrum.
Thus, $L$ has in some sense the worst possible spectral structure that still allows for the 
generation of a semigroup. This makes the analysis of $S(t)$ mathematically interesting
and challenging since the application of general ``soft'' arguments is largely precluded.
Furthermore, besides the well-known applications to probability and mathematical finance,
operators of the form \eqref{eq:L} occur very naturally in the study of self-similar
solutions to nonlinear parabolic equations.
To see this, consider for instance the equation
\[ \partial_t u(t,x)=\Delta_x u(t,x)+F(u(t,x),|x|) \]
where $F$ is some given nonlinearity that allows for the existence of a radial self-similar
solution of the form $u_0(t,x)=(1-t)^{-\beta} f(\frac{|x|}{\sqrt{1-t}})$.
In order to analyse the stability of $u_0$, it is standard to introduce
\emph{similarity coordinates} $\tau=-\log(1-t)$, $\xi=\frac{x}{\sqrt{1-t}}$.
If $F$ scales suitably, the change of variables 
$(t,x)\mapsto (\tau,\xi)$
leads to an equation of the form
\begin{align*} \partial_\tau \tilde u(\tau,\xi)
=&\Delta_\xi \tilde u(\tau,\xi)-\tfrac12 \xi\cdot \nabla_\xi 
\tilde u(\tau,\xi)-\beta \tilde u(\tau,\xi) \\
&+\partial_1 F\left (f(|\xi|),|\xi|\right )\tilde u(\tau,\xi)+\mbox{nonlinear terms}
\end{align*}
for $\tilde u(\tau,\xi)=(1-t)^\beta[u(t,x)-u_0(t,x)]$. 
Consequently,
the linear part on the right-hand side is 
an Ornstein-Uhlenbeck
operator as in \eqref{eq:L}.
In such a situation one is naturally led to the unweighted setting since
Sobolev spaces with \emph{decaying} Gau\ss ian weights are in general
not suitable to study nonlinear
problems. 

As usual, the important question for applications is whether one can derive
growth estimates
for the semigroup $S(t)$ by merely looking at the spectrum of $L$, which is typically the only
accessible information.
In the present paper we answer this question in the affirmative for
the class of operators defined by Eq.~\eqref{eq:L}.
We prove the strongest possible result in this context, namely that the spectrum
of $S(t)$ is completely determined by the spectrum of $L$.

\begin{theorem}[Spectral mapping for Ornstein-Uhlenbeck operators]
\label{thm:main}
Suppose $V: [0,\infty)\to \C$ satisfies $|V(r)|\leq C\langle r\rangle^{-2}$, 
$|V'(r)|\leq C\langle r\rangle^{-3}$ for all $r\geq 0$ and some constant $C>0$.
Set
\[ \mc D(\tilde L):=\{u\in H^2(\R^d)\cap C^\infty(\R^d): \mc L_V u \in L^2(\R^d)\}, \]
and define $\tilde L: \mc D(\tilde L)\subset L^2(\R^d)\to L^2(\R^d)$ by $\tilde Lu:=\mc L_V u$.
Then $\tilde L$ is densely defined, closable, and its closure $L$ generates a 
strongly continuous one-parameter
semigroup $\{S(t): t\geq 0\}$ of bounded operators on $L^2(\R^d)$ such that the spectral mapping
\begin{equation*}
 \sigma(S(t))\backslash \{0\}=\{e^{t\lambda}: \lambda \in \sigma(L)\} 
 \end{equation*}
holds for all $t\geq 0$.
\end{theorem}

\subsection{Remarks}
We remark that the proof of 
Theorem \ref{thm:main} considerably simplifies if $V\leq 0$ because in this case one has
$\sigma(L)=\sigma(L_0)$ and from \cite{Met01} it follows that $\sigma(L_0)=\{z\in \C: \Re z\leq d\}$,
see also Lemma \ref{lem:specL0} below.
However, as explained in the introduction,
our main motivation for studying this problem comes from the stability of self-similar
solutions. In this context there is no reason to believe
that the potential has a sign. 
In fact, the most interesting situations 
occur if there exist self-similar solutions
with a finite number of unstable modes.
The corresponding potentials will then have finitely many zeros.
Since our proof works equally well for complex-valued potentials,
we decided to state Theorem \ref{thm:main} in this generality.

Determining the spectrum of $L$ is a different problem which we do not touch upon in this
paper. After all, the spectrum of $L$ depends on the concrete form
of $V$.
On the other hand, as a by-product of our investigations, we can at least deduce the following 
nontrivial property.

\begin{theorem}
\label{thm:main1}
Let $\sigma'(L):=\sigma(L)\backslash \sigma(L_0)$. Then, for any $b>0$, the set
\[ \sigma'(L)\cap \{d+b+\I\omega: \omega\in \R\} \]
is bounded.
\end{theorem}

\noindent In this respect we also remark that our construction seems to imply that
$\sigma'(L)$ 
is discrete
since the addition of the potential $V$ does not change the asymptotics of the involved
ODEs. 
However, we do not elaborate on this any further 
since in typical applications  
the potential $V$ induces a relatively compact operator and the abstract
theory implies that 
$\sigma'(L)$ consists of eigenvalues only. 
As a matter of fact, relative compactness can already be deduced from very mild 
decay properties of the potential, see e.g.~\cite{HanHol10}.

Finally, we would like to mention that it is possible to weaken the assumptions on $V$ considerably.
For instance, inspection of the proof shows that $V\in W^{1,1}(\R_+)$ suffices for
the argument to go through. However, for the sake of simplicity we do not prove
Theorem \ref{thm:main} in this generality.

\subsection{Further discussion and related work}
Ornstein-Uhlenbeck  operators are mostly studied in suitable weighted spaces
with invariant measures, e.g.~\cite{Lun97, MetPalPri02, MetPruRhaSch02, KojYok10, ForRha13, BerLor07}.
However, there is also a growing interest in the corresponding operators
acting on spaces with more general weights \cite{StuerzerArnold2014} or on unweighted $L^p$-spaces as in the present paper, 
see e.g.~\cite{Met01, AreMetPal06, HieLorPruRhaSch09, MetOuhPal11}.

In general, the question of spectral mapping between the semigroup and its generator
is of uttermost importance for applications since typically, the only way to 
determine the stability of a time-evolution system is to study the spectrum of its generator.
Unfortunately, spectral mapping is not stable under bounded perturbations.
If the perturbed semigroup does not have ``nice'' properties such as eventual norm continuity,
it can be very difficult to prove a suitable spectral mapping theorem.
Parabolic equations in 
non-self-adjoint settings are a prominent example 
where spectral mapping is nontrivial, see e.g.~\cite{Helffer2013}
and references therein.
Furthermore, hyperbolic equations are an important class of evolution problems where 
difficult problems related to spectral mapping occur
since the spectrum
contains the imaginary axis. 
Although there are many positive results, 
see e.g.~\cite{GesJonLatSta00} for the case of Schr\"odinger
equations and \cite{Lic08, BatLiaXia05} for hyperbolic systems, it was precisely in this context where
Renardy constructed his by now famous counterexample \cite{Ren94}. It shows 
that very natural, relatively compact perturbations of the wave
equation can destroy spectral mapping, even in a standard $L^2$-space.
This simple example came as a shock although many counterexamples to spectral mapping 
were known at the time,
e.g.~\cite{HilPhi57, Foi73, Zab75, HejKap86}. However, there
was a widespread belief that such a pathology is confined to rather artificially
constructed situations that never occur in real-world applications.
In view of this, Renardy's example is very disturbing.  On the other hand, 
it is known that in Hilbert spaces ``most'' bounded perturbations preserve spectral mapping
\cite{Ren96}. This is a positive result from a psychological point of view 
but it cannot be used to
deduce spectral mapping for a concrete problem.

To the knowledge of the authors, the most general ``abstract'' conditions 
that guarantee spectral mapping are based on norm continuity properties, see \cite{BreNagPol00}. 
However, we do not see how to apply the theory from \cite{BreNagPol00} 
to the problem at hand, see
Appendix \ref{sec:apx} for a discussion on this.
That is why we choose a more explicit approach.
The key tool 
is the Gearhart-Pr\"u\ss{} Theorem \cite{Gear78, Pru84} 
which reduces the question of spectral mapping 
to uniform bounds on the resolvent with respect to the imaginary part of the spectral parameter.
Consequently, we perform an explicit 
perturbative construction of 
the resolvent along vertical lines in the complex plane, for large imaginary parts
of the spectral parameter. 
This is possible because the potential is assumed to be radial which allows us
to reduce the spectral problem to an infinite number of decoupled ODEs, one for
each value of the angular momentum parameter $\ell$.
For each fixed $\ell$ we construct the ``reduced'' resolvent by asymptotic ODE methods
based on the Liouville-Green transform, and we establish $L^2$-bounds that hold
uniformly in $\ell$.
These allow us to obtain the desired bounds for the ``full'' resolvent by summing
over $\ell$ and spectral
mapping follows from the Gearhart-Pr\"u\ss{} Theorem.

\subsection{Notation}
We use standard Lebesgue and Sobolev spaces denoted by 
$L^p$, $W^{k,p}$, $H^k=W^{k,2}$, and $\mc S$ is the Schwartz space.
Furthermore, for $f: \R_+\to \C$ we write
\[ \|f\|_{L_\mathrm{rad}^2(\R^d)}^2:=\int_0^\infty |f(r)|^2 r^{d-1}dr \]
where $\R_+:=[0,\infty)$.
The letter $C$ (possibly with indices to indicate dependencies) denotes
a positive constant that might change its value at each occurrence.
We write $f(r)=O_\C (g(r))$ if 
$|f(r)|\leq C|g(r)|$ \emph{and} $|f'(r)|\leq C|g'(r)|$ (the subscript $\C$ indicates
that $f$ might be complex-valued).
As usual, $A \lesssim B$ means $A\leq CB$ where $C$ is independent of all the parameters
that occur in the inequality.
We also write $A\simeq B$ if $A\lesssim B$ and $B\lesssim A$.
$A\gg B$ means that $A\geq CB$ for $C$ sufficiently large.
Furthermore, we frequently use the ``japanese bracket'' notation $\langle x\rangle:=\sqrt{1+|x|^2}$.
For the Wronskian $W(f,g)$ of two functions $f$ and $g$ we use the convention
$W(f,g)=fg'-f'g$.
The domain of a closed operator $A$ on a Banach space is denoted by $\mc D(A)$ and we write
$\sigma(A)$ for its spectrum and $\rho(A)=\C\backslash \sigma(A)$ for its resolvent set.
Finally, in the technical part we restrict ourselves to $d\geq 3$ in order to avoid technicalities 
involving logarithmic
corrections. With minor modifications, the same construction can be performed in the case $d=2$.

\section{Preliminaries}

\subsection{Generation of the semigroup}
For the sake of completeness we include the generation result which is of course well 
known \cite{BerLor07}.

\begin{lemma}
\label{lem:gen}
Define $\tilde L_0: \mc D(\tilde L)\subset L^2(\R^d)\to L^2(\R^d)$ by $\tilde L_0 u:=\tilde Lu-Vu$
where $\tilde L$ and $V$ are from Theorem \ref{thm:main}.
Then the operator $\tilde L_0$ is densely defined, closable, and its closure $L_0$ 
generates a strongly continuous one-parameter semigroup 
$\{S_0(t): t\geq 0\}$ which satisfies
\[ \|S_0(t)\|_{L^2(\R^d)}\leq e^{dt} \]
for all $t\geq 0$. As a consequence, $L$ generates a 
semigroup $\{S(t): t\geq 0\}$ satisfying
\[ \|S(t)\|_{L^2(\R^d)}\leq e^{Mt} \]
for all $t\geq 0$ where $M=d+\|V\|_{L^\infty(\R_+)}$.
\end{lemma}

\begin{proof}
Since $C^\infty_c(\R^d)\subset \mc D(\tilde L)$, 
it is obvious that $\tilde L_0$ is densely defined and integration by parts yields
\[ \Re (\tilde L_0 u | u)_{L^2(\R^d)}=-\int_{\R^d}|\nabla u(x)|^2 dx
+d\int_{\R^d}|u(x)|^2 dx\leq d \|u\|_{L^2(\R^d)}^2 \]
for all $u\in \mc D(\tilde L)$ (the boundary terms vanish by the density of $C^\infty_c(\R^d)$
in $H^2(\R^d)$).
Now we claim that the operator $2d-\tilde L_0$ has dense range.
To prove this, it suffices to show that the equation $(2d-\tilde L_0)u=f$ has a solution 
$u\in H^2(\R^d)\cap C^\infty(\R^d)$ for any given $f\in \mc S(\R^d)$.
On the Fourier side this equation reads
\[ (|\xi|^2-2\xi\cdot \nabla)\hat u(\xi)=\hat f(\xi) \]
which is solved by
\[ \hat u(r\omega)=\int_r^\infty \frac{e^{-\frac14(s^2-r^2)}}{2s}\hat f(s\omega)ds \]
where we introduced polar coordinates $r=|\xi|$ and $\omega=\xi/|\xi|$.
We obtain
\[ \hat u(r\omega)r^{\frac{d-1}{2}}=\int_0^\infty K(r,s)\hat f(s\omega)s^{\frac{d-1}{2}}ds \]
with $K(r,s)=\frac12 r^{\frac{d-1}{2}}e^{-\frac14(s^2-r^2)}s^{-\frac{d+1}{2}}1_{[0,\infty)}(s-r)$.
The kernel $K$ satisfies the bound $|K(r,s)|\lesssim \min\{r^{-1},s^{-1}\}$ and hence,
it induces a bounded operator on $L^2(\R_+)$ (see e.g.\ \cite{DonKri13}, Lemma 5.5).
Consequently, we infer
\begin{align*} 
\|\hat u\|_{L^2(\R^d)}^2&=\int_{\S^{d-1}}\|\hat u(\cdot\, \omega)\|_{L^2_\mathrm{rad}(\R^d)}^2
d\sigma(\omega)
\lesssim \int_{\S^{d-1}} \|\hat f(\cdot\, \omega)\|_{L^2_\mathrm{rad}(\R^d)}^2 d\sigma(\omega) \\
&=\|\hat f\|_{L^2(\R^d)}^2.
\end{align*}
Furthermore, on the support of the kernel $K(r,s)$ 
we have $r\leq s$
and thus, by the same reasoning as above we obtain the bound 
\[ \|\langle \cdot\rangle^N \hat u\|_{L^2(\R^d)}\leq C_N \|\langle \cdot \rangle^N \hat f
\|_{L^2(\R^d)} \]
for any $N\in \N$ and the right-hand side is finite since $\hat f$ is Schwartz.
By taking the inverse Fourier transform of $\hat u$,
we obtain a function $u\in H^2(\R^d)\cap C^\infty(\R^d)$ which satisfies $(2d-\tilde L_0)u=f$.
Consequently, the Lumer-Phillips Theorem yields the existence
of the semigroup
$\{S_0(t):t\geq 0\}$ with the stated bound.
The statements for $L$ and $S(t)$ follow from the Bounded Perturbation Theorem.
\end{proof}

We also recall the spectral structure of the free Ornstein-Uhlenbeck operator $L_0$.
Note that the following lemma shows in particular that the growth bound
$\|S_0(t)\|_{L^2(\R^d)}\leq e^{dt}$ is sharp and that $S$ is a general $C_0$-semigroup without
additional properties (such as eventual norm-continuity).

\begin{lemma}
\label{lem:specL0}
The spectrum of $L_0$ (defined in Lemma \ref{lem:gen}) is given by
\[ \sigma(L_0)=\{z\in \C: \Re z\leq d\}. \]
Furthermore, we have $\sigma(L_0)\subset \sigma(L)$.
\end{lemma}

\begin{proof}
The statement on $\sigma(L_0)$ follows from \cite{Met01} and $\sigma(L_0)\subset \sigma(L)$
is proved in Appendix \ref{sec:specL0}.
\end{proof}

\subsection{Angular decomposition of the resolvent}
From now on we set $R(\lambda):=(\lambda-L)^{-1}$ for $\lambda \in \rho(L)$ where $L$
is from Theorem \ref{thm:main}.
We exploit the assumed radial symmetry of the potential $V$ by an angular decomposition.
Let $Y_{\ell,m}: \S^{d-1}\to \C$ denote a standard spherical harmonic, i.e., an
$L^2(\S^{d-1})$-normalized eigenfunction
of the Laplace-Beltrami operator on $\S^{d-1}$ with eigenvalue $\ell(\ell+d-2)$. 
We denote by $\Omega_d \subset \N_0\times \Z$ the set of admissible values of $(\ell,m)$.
The precise domain of $m$ is irrelevant for us but we note that all $\ell\in \N_0$ occur.
For $f\in L^2(\R^d)$ we define 
\[ [P_{\ell,m}f](r):=\int_{\S^{d-1}}f(r\omega)\overline{Y_{\ell,m}(\omega)}d\sigma(\omega) \]
and Cauchy-Schwarz implies that $P_{\ell,m}$ is a bounded operator from $L^2(\R^d)$
to $L_\mathrm{rad}^2(\R^d)$ with operator norm $1$.
Furthermore, we define a bounded operator $Q_{\ell,m}: L^2_\mathrm{rad}(\R^d)\to 
L^2(\R^d)$ by
\[ [Q_{\ell,m}g](x):=g(|x|)Y_{\ell,m}(\tfrac{x}{|x|}) \]
and again, the operator norm of $Q_{\ell,m}$ is $1$.
We also note that $Q_{\ell,m}P_{\ell,m}: L^2(\R^d)\to L^2(\R^d)$ is a projection.
The fact that the formal differential operator $\mc L_V$ in polar coordinates
 separates into a radial and an angular component can now be phrased in
operator language as follows.

\begin{lemma}
\label{lem:commute}
Let $\lambda \in \rho(L)$. Then $[R(\lambda), Q_{\ell,m}P_{\ell,m}]=0$ for all $(\ell,m)\in \Omega_d$.
\end{lemma}

\begin{proof}
It suffices to prove the identity $R(\lambda) Q_{\ell,m}P_{\ell,m}=Q_{\ell,m}P_{\ell,m} R(\lambda)$ on a dense subset of
$L^2(\R^d)$. So let $f\in C^\infty_c(\R^d)$.
By Lemma \ref{lem:specL0} we have $\lambda \in \rho(L_0)$ and we set $u=R_0(\lambda) f$
where $R_0(\lambda):=(\lambda-L_0)^{-1}$ is the resolvent of the free Ornstein-Uhlenbeck
operator $L_0$.
By elliptic regularity we have $u\in C^\infty(\R^d)$ and we infer
\begin{align*} 
[P_{\ell,m}L_0 u](r)&=\int_{\S^{d-1}}(D_r+\tfrac{1}{r^2}\Delta_\omega)
u(r\omega)\overline{Y_{\ell,m}(\omega)}d\sigma(\omega), \qquad r>0
\end{align*}
where $D_r=\partial_r^2+\frac{d-1}{r}\partial_r-2r\partial_r$ and $-\Delta_\omega$ is
the Laplace-Beltrami operator on $\S^{d-1}$.
By dominated convergence and the fact that $\Delta_\omega$ is symmetric on $L^2(\S^{d-1})$
we infer
\begin{align*} [P_{\ell,m}L_0 u](r)&=\left [D_r-\tfrac{\ell(\ell+d-2)}{r^2}\right ] 
\int_{\S^{d-1}}u(r\omega)
\overline{Y_{\ell,m}(\omega)}d\sigma(\omega) \\
&=\left [D_r-\tfrac{\ell(\ell+d-2)}{r^2}\right ]P_{\ell,m}u(r)
\end{align*}
and this implies
\[ Q_{\ell,m}P_{\ell,m}(\lambda-L_0) u=(\lambda-L_0)Q_{\ell,m}P_{\ell,m}u. \]
Consequently, we obtain $R_0(\lambda)Q_{\ell,m}P_{\ell,m}f=Q_{\ell,m}P_{\ell,m}R_0(\lambda)f$.
The claimed $[R(\lambda),Q_{\ell,m}P_{\ell,m}]=0$ follows now from the identity
\[ R(\lambda)=R_0(\lambda)[1-VR_0(\lambda)]^{-1} \] and the fact that $V$ is radial
(invertibility of $1-VR_0(\lambda)$ is a consequence of
$\lambda-L=[1-VR_0(\lambda)](\lambda-L_0)$ and $\lambda\in \rho(L_0)\cap \rho(L)$).
\end{proof}

\begin{definition}For $\lambda \in \rho(L)$ and $(\ell,m)\in \Omega_d$ we define the \emph{reduced resolvent} 
$R_{\ell,m}(\lambda): L^2_\mathrm{rad}(\R^d)\to L^2_\mathrm{rad}(\R^d)$ by
\[ R_{\ell,m}(\lambda):=P_{\ell,m}R(\lambda)Q_{\ell,m}. \]
\end{definition}

\begin{lemma}
\label{lem:reduced}
For every $\lambda \in \rho(L)$ we have the bound
\[ \|R(\lambda)\|_{L^2(\R^d)}\leq 
\sup_{(\ell,m)\in \Omega_d}\|R_{\ell,m}(\lambda)\|_{L^2_\mathrm{rad}(\R^d)}. \]
\end{lemma}

\begin{proof}
For brevity we write $\sum_{\ell,m}:=\sum_{(\ell,m)\in \Omega_d}$.
Every $f\in L^2(\R^d)$ can be expanded as 
\[ f=\sum_{\ell,m}Q_{\ell,m}P_{\ell,m}f \]
and the expansion converges in $L^2(\R^d)$.
Moreover, Parseval's identity and the monotone convergence theorem yield
\[ \|f\|_{L^2(\R^d)}^2=\sum_{\ell,m}\|P_{\ell,m}f\|_{L_\mathrm{rad}^2(\R^d)}^2. \]
Consequently, by Lemma \ref{lem:commute} and $(Q_{\ell,m}P_{\ell,m})^2=Q_{\ell,m}P_{\ell,m}$
we infer
\begin{align*}
R(\lambda)f&=\sum_{\ell,m}Q_{\ell,m}P_{\ell,m}R(\lambda)f
=\sum_{\ell,m}Q_{\ell,m}P_{\ell,m}R(\lambda)Q_{\ell,m}P_{\ell,m}f \\
&=\sum_{\ell,m}Q_{\ell,m}R_{\ell,m}(\lambda)P_{\ell,m}f
\end{align*}
which implies
\begin{align*} 
\|R(\lambda)f\|_{L^2(\R^d)}^2&\leq \sum_{\ell,m}\|R_{\ell,m}(\lambda)P_{\ell,m}f\|_{L^2_\mathrm{rad}(\R^d)}^2  \\
&\leq \sup_{\ell,m}\|R_{\ell,m}(\lambda)\|_{L_\mathrm{rad}^2(\R^d)}^2\sum_{\ell,m}\|P_{\ell,m}
f\|_{L_\mathrm{rad}^2(\R^d)}^2 \\
&\leq \sup_{\ell,m}\|R_{\ell,m}(\lambda)\|_{L_\mathrm{rad}^2(\R^d)}^2 \|f\|_{L^2(\R^d)}^2.
\end{align*}
\end{proof}

\noindent The main result of the present paper is in fact the 
following estimate on the reduced resolvent.

\begin{theorem}
\label{thm:main2}
Let $b>0$.
Then the reduced resolvent $R_{\ell,m}$ satisfies
\[ \|R_{\ell,m}(d+b+\I\omega)\|_{L_\mathrm{rad}^2(\R^d)}\leq C \]
for all $\omega \gg 1$ and all $(\ell,m) \in \Omega_d$.
\end{theorem}

\subsection{Reduction of Theorems \ref{thm:main} and \ref{thm:main1} to Theorem \ref{thm:main2}}
Now we show that Theorem \ref{thm:main2} implies Theorems \ref{thm:main}
and \ref{thm:main1}.
The rest of the paper is then devoted to the proof of Theorem \ref{thm:main2}.

\begin{lemma}
\label{lem:Restomain}
Theorem \ref{thm:main2} implies Theorems \ref{thm:main} and \ref{thm:main1}.
\end{lemma}

\begin{proof}
First of all we note that by complex conjugation, the stated resolvent bound 
in Theorem \ref{thm:main2} holds for
large negative $\omega$ as well.
We may assume $t>0$ and
use the common abbreviation
\[ e^{t\sigma(L)}:=\{e^{tz}: z\in \sigma(L)\}. \]
Recall that the inclusion $e^{t\sigma(L)}\subset \sigma(S(t))$ always holds.
Thus, in order to show Theorem \ref{thm:main},
it suffices to prove that $\C^*\backslash e^{t\sigma(L)}\subset \rho(S(t))$,
where $\C^*:=\C\backslash \{0\}$.
Now assume that $\lambda\in \C^*\backslash e^{t\sigma(L)}$
and suppose $\lambda=e^{tz}$ for some $z\in \C$.
Then we must have $z\in \rho(L)$ since otherwise $\lambda \in e^{t\sigma(L)}$. 
Thus, we obtain
\begin{equation*}
 \{z\in \C: e^{tz}=\lambda\}\subset \rho(L). 
 \end{equation*}
 Furthermore, Lemma \ref{lem:specL0} shows that $\frac{1}{t}\log |\lambda|>d$ and
by assumption and Lemma \ref{lem:reduced} we have the bound
\begin{equation*}
 \|R(\tfrac{1}{t}\log |\lambda|+\tfrac{\I}{t}\arg\lambda+\tfrac{2\pi \I k}{t})\|_{L^2(\R^d)}\leq C 
 \end{equation*}
for all $k\in \Z$ (if $\frac{1}{t}\log|\lambda|>d+\|V\|_{L^\infty(\R_+)}$, the stated bound follows directly
from the growth bound in Lemma \ref{lem:gen}).
Consequently, the set
\[ \{\|R(z)\|_{L^2(\R^d)}: e^{tz}=\lambda\}\subset \R \]
is bounded and the Gearhart-Pr\"u\ss{} Theorem (see \cite{Pru84}, Theorem 3) yields
$\lambda \in \rho(S(t))$.
This proves Theorem \ref{thm:main}.

To prove Theorem \ref{thm:main1}, it suffices to note that, for given $b>0$, 
it is a consequence of the expansion in the proof of Lemma \ref{lem:reduced}
and Theorem \ref{thm:main2} that the resolvent $R(d+b+\I \omega)$ exists,
provided $|\omega|$ is large enough.
This implies Theorem \ref{thm:main1}.
\end{proof}

\section{A nontechnical outline of the resolvent construction}
\noindent Since the rest of the paper is very technical, we outline the main steps of our construction
in a less formal fashion.
This should aid the interested reader when going through the details of the proof.

\subsection{Setup}
Our goal is to construct the reduced resolvent $R_{\ell,m}(d+b+\I\omega)$, at least
for large values of $\omega$.
If we set $u_{\ell,m}:=R_{\ell,m}(d+b+\I\omega)f$, then $u_{\ell,m}$ satisfies the ODE
\begin{align}
\label{eq:out_u}
u_{\ell,m}''(r)&+\tfrac{d-1}{r}u_{\ell,m}'(r)-2 r u_{\ell,m}'(r)
-\tfrac{\ell(\ell+d-2)}{r^2}u_{\ell,m}(r) \nonumber \\
&+V(r)u_{\ell,m}(r)-\lambda u_{\ell,m}(r)=-f_{\ell,m}(r)
\end{align}
where $f_{\ell,m}(r)=(f(r\,\cdot)|Y_{\ell,m})_{\S^{d-1}}$.
It is convenient to transform Eq.~\eqref{eq:out_u} to normal form which is achieved by
setting $u_{\ell,m}(r)=r^{-\frac{d-1}{2}}e^{r^2/2}v(r)$, where we suppress now the subscripts
$\ell,m$ on $v$ in order to avoid notational clutter.
We obtain
\begin{align}
\label{eq:out_v}
v''(r)&-r^2 v(r)-\tfrac{(d+2\ell-1)(d+2\ell-3)}{4r^2}v(r)
-(\lambda-d)v(r) \nonumber \\
&+V(r)v(r)=-r^{\frac{d-1}{2}}e^{-r^2/2}f_{\ell,m}(r)
\end{align}
and the goal is to invert this operator, i.e., we have to compute $v$ in terms of
$f_{\ell,m}$.
By the variation of constants formula, $v$ is given by
\[ v(r)=\int_0^\infty G(r,s;\lambda,\ell)s^{\frac{d-1}{2}}e^{-s^2/2}f_{\ell,m}(s)ds \]
where $G$ is the Green function, i.e.,
\[ G(r,s;\lambda,\ell)=-\frac{1}{W(v_0,v_-)(\lambda,\ell)}\left \{
\begin{array}{lr}v_0(r)v_-(s) & \mbox{ if }r\leq s \\
v_0(s)v_-(r) & \mbox{ if }r \geq s \end{array} \right . . \]
Here, $\{v_0,v_-\}$ is a fundamental system for the \emph{homogeneous equation}, that
is, Eq.~\eqref{eq:out_v} with $f_{\ell,m}=0$.
The functions $v_0$ and $v_-$ need to be chosen in such a way that $v_0$ is 
the ``good'' solution near $0$ and $v_-$ is the ``good'' solution near $\infty$.

The difficulty in obtaining the necessary estimates on $G$ is the fact
that $v_0$ and $v_-$ depend on $\ell$ and $\lambda$ which
are two potentially large parameters.
Thus, we need uniform control of $v_0$ and $v_-$ for all $\ell\in \N$ and 
$\lambda=d+b+\I\omega$
with $\omega$ large (the parameter $b$ is fixed).
This is a challenging two-parameter asymptotic problem.

Since the potential $V$ is not known, one cannot hope for explicit expressions for the
functions $v_0$ and $v_-$.
Consequently, throughout the paper we treat the potential perturbatively, that is,
we rewrite the homogeneous version of Eq.~\eqref{eq:out_v} as
\begin{align}
\label{eq:out_vhom}
v''(r)&-r^2 v(r)-\tfrac{\nu^2-\frac14}{r^2}v(r)
-\mu v(r) =-V(r)v(r)
\end{align}
and for brevity we introduce the parameters $\mu=\lambda-d=b+\I\omega$ and $\nu=\frac{d}{2}+\ell-1$.
Our hope is that the right-hand side of Eq.~\eqref{eq:out_vhom} is in some sense
negligible 
if $|\mu|$ is large (this is the only case we are interested in). 
Consequently, the first step is to solve Eq.~\eqref{eq:out_vhom} with $V=0$.
Although this equation can be solved explicitly in terms of parabolic cylinder functions,
it turns out that the corresponding expressions are still too complicated to proceed.
Thus, we do not rely on any kind of asymptotic theory for parabolic cylinder functions
but choose a different approach which we now explain.

\subsection{The Liouville-Green transform}
\label{sec:out_LG}
It is expected that the only relevant property of solutions 
to Eq.~\eqref{eq:out_vhom} (with $V=0$) is
their asymptotic behavior as $r\to 0+$ and $r\to \infty$, 
which cannot be terribly complicated.
Consequently, it should be possible to add a correction potential $\tilde Q$ (depending
on $\omega$ and $\nu$, of course) to both sides
of Eq.~\eqref{eq:out_vhom} such that
\[ v''(r)-r^2 v(r)-\tfrac{\nu^2-\frac14}{r^2}v(r)
-\mu v(r)+\tilde Q(r)v(r)=0 \]
has a ``simple'' fundamental system and the ``new'' right-hand side 
\[ -V(r)v(r)+\tilde Q(r)v(r) \]
can still be treated perturbatively.
The technical device that is used to achieve this is the Liouville-Green transform.
We briefly recall how it works which is most easily done by considering a toy problem.

Suppose we are given an equation of the form
\begin{equation}
\label{eq:out_toy}
 f''(x)+q(x)f(x)+af(x)=0 
 \end{equation}
where $a>0$ is a potentially large parameter.
The transformation $g(\varphi(x))=\varphi'(x)^\frac12 f(x)$ for an orientation-preserving 
diffeomorphism $\varphi$ yields
\[ g''(\varphi(x))+\frac{q(x)+a}{\varphi'(x)^2}g(\varphi(x))-\frac{Q(x)}{\varphi'(x)^2}
g(\varphi(x))=0 \]
where
\[ Q(x)=\frac12 \frac{\varphi'''(x)}{\varphi'(x)}-\frac34 \frac{\varphi''(x)^2}{\varphi'(x)^2} \]
is called the Liouville-Green potential.
So far, this is a general observation.
The transformation becomes useful only if one chooses $\varphi$ in a clever way, depending
on what kind of information one would like to obtain.
For instance, if it is possible to choose $\varphi$ 
in such a way that $\frac{q(x)+a}{\varphi'(x)^2}=a$, one sees that the equation
\[ f''(x)+q(x)f(x)+af(x)+Q(x;a)f(x)=0 \]
has the solutions $\varphi'(x)^{-\frac12}e^{\pm \I a \varphi(x)}$. (of course,
$\varphi$ depends on $a$
but we suppress this).
Thus, one rewrites Eq.~\eqref{eq:out_toy}
as
\[  f''(x)+q(x)f(x)+af(x)+Q(x;a)f(x)=Q(x;a)f(x) \]
and if $Q$ is small for $x$ large and $a$ large, say, one may obtain solutions of
Eq.~\eqref{eq:out_toy} of the form 
$\varphi'(x)^{-\frac12}e^{\pm \I a \varphi(x)}[1+\varepsilon_\pm(x,a)]$
where $\varepsilon_\pm(x,a)$ goes to zero as $x\to \infty$ or $a\to \infty$.
The analysis of solutions of Eq.~\eqref{eq:out_toy} is then reduced to the analysis
of the function $\varphi$ which may be considerably easier.

\subsection{Volterra iterations}
\label{sec:out_Volt}
Next, we describe the perturbative treatment of the right-hand side based on Volterra
iterations. 
For simplicity, we stick to the above toy problem Eq.~\eqref{eq:out_toy}
and set $\psi_\pm(x;a):=\varphi'(x)^{-\frac12}e^{\pm \I a \varphi(x)}$.
Suppose we would like to construct a solution to Eq.~\eqref{eq:out_toy} of the form 
$\psi_-(x;a)
[1+\varepsilon_-(x;a)]$ where $\varepsilon_-(x;a)$ vanishes as $x\to\infty$.
The variation
of constants formula yields the integral equation
\begin{equation}
\label{eq:out_volt}
h(x;a)=1+\int_x^\infty K(x,y;a)h(y;a)dy
\end{equation}
for the function $h=1+\varepsilon_-$, where
\[ K(x,y;a)=\frac{1}{W(a)}\left [\psi_-\psi_+(y;a)-\frac{\psi_+}{\psi_-}(x;a)\psi_-(y;a)^2
\right ]Q(y;a) \]
and $W(a)=W(\psi_-(\cdot;a),\psi_+(\cdot;a))=-2\I a$.
Now suppose 
\[ m_0:=\int_{x_0}^\infty \sup_{x \in (x_0,y)}|K(x,y;a)|dy<\infty \]
for some $x_0$ and all $a\geq 1$.
Then the basic theorem on Volterra equations, see e.g.~\cite{SchSofSta10}, shows that
Eq.~\eqref{eq:out_volt} has a solution $h$ 
that satisfies $|h(x;a)|\leq e^{m_0}$ for all $x\geq x_0$ and $a\geq 1$.
In a typical situation (for instance if $Q(x;a)$ decays like $\langle x\rangle^{-2}$)
one has an estimate like
\[ |K(x,y;a)|\lesssim \langle y\rangle^{-2}a^{-1} \]
for all $x_0\leq x\leq y$ and $a\geq 1$.
This bound implies the existence of 
$h$ and the Volterra equation yields the decay
\[ |\varepsilon_-(x;a)|=|h(x;a)-1|\leq e^{m_0} \int_x^\infty |K(x,y;a)|dy
\lesssim \langle x\rangle^{-1}a^{-1}. \]
In this way one would obtain a solution to Eq.~\eqref{eq:out_toy}
of the form $\varphi'(x)^{-\frac12}e^{-\I a \varphi(x)}[1+O(\langle x \rangle^{-1}a^{-1})]$
and in order to prove bounds that hold uniformly for large $x$ \emph{and} large $a$, it 
again suffices to study the function $\varphi$.

Another nice feature of Volterra iterations is the fact that the constructed functions
inherit differentiability properties of the potential.
In a typical situation the potential satisfies symbol-type bounds of the form
\[ |\partial_x^k Q(x;a)|\leq C_k \langle x\rangle^{-2-k},\quad k\in \N_0 \]
and these types of bounds are usually inherited by the function $h$, cf.~Remark 
\ref{rem:Volterra} below.

In the technical part, all our perturbative arguments are based on this
scheme and we make free use of the above observations.

\subsection{Construction of fundamental systems}
After this interlude we return to Eq.~\eqref{eq:out_vhom}.
In order to apply the machinery described in Sections \ref{sec:out_LG} and \ref{sec:out_Volt},
it is necessary to distinguish a number of cases which we name after the approximating
equations.
Recall that the relevant parameters are $\nu=\frac{d}{2}+\ell-1$ and $\omega$, where
throughout,
$\mu=\lambda-d=b+\I\omega$.
We are only interested in $\omega$ large (which we always assume) whereas $\nu$ can
be small or large.
The parameter $b$ is fixed and thus not relevant.
As a consequence, all implicit constants are allowed to depend on $b$
(but, of course, \emph{not} on $\nu$ or $\omega$).
Furthermore, we have the variable $r$ which can be small or large.
Depending on the relative location of $r$, $\omega$, and $\nu$, 
we move different terms in Eq.~\eqref{eq:out_vhom} to the right-hand side, apply
a suitable Liouville-Green transform, and perform a perturbative construction
as outlined in Sections \ref{sec:out_LG} and \ref{sec:out_Volt}.
\begin{enumerate}
\item Small angular momenta: $\nu\leq \nu_0$ for some sufficiently large absolute
constant $\nu_0>0$.
\begin{enumerate}
\item Weber case: $r\geq 1$, Proposition \ref{prop:fsrbig}. We rewrite Eq.~\eqref{eq:out_vhom} as
\begin{align*} v''(r)&-r^2 v(r)-\tfrac{\nu^2}{r^2}v(r)
-\mu v(r) \\
&=-\tfrac{1}{4r^2}v(r)-V(r)v(r). 
\end{align*}
The dominant contribution comes from a Weber-type equation. We construct a fundamental system
$\{v_-,v_+\}$ of the form
\[ v_{\pm}(r,\omega)=\tfrac{1}{\sqrt{2}}\mu^{-\frac14}\xi'(\mu^{-\frac12}r)^{-\frac12}
e^{\pm \mu\xi(\mu^{-1/2}r)}[1+\varepsilon_{\pm}(r,\omega)] \]
where $\varepsilon_{\pm}$ are complex-valued functions that satisfy the bound
\begin{align*} 
|\varepsilon_\pm(r,\omega)|\leq C r^{-1}\omega^{-\frac12} 
\end{align*}
for all $r\geq 1$, $\omega\geq \omega_0$ (where $\omega_0$ is a sufficiently large
absolute constant), $\nu\geq 0$, and some absolute constant $C>0$. 
Of course, $v_\pm$ and $\varepsilon_\pm$ also depend on $\nu$ but in the domain
$\nu\leq \nu_0$ this dependence is inessential and we suppress it in the notation.
In the sequel we will use the compact notation 
$\varepsilon_\pm(r,\omega)=O_\C(r^{-1}\omega^{-\frac12}\nu^0)$ which carries
all the relevant information.
The function $\xi$ is given in closed form as an integral and its analysis is in principle
straightforward.

\item Hankel case: $c\omega^{-\frac12}\leq r\leq 1$ where $c>0$ is a sufficiently large absolute
constant, Lemma \ref{lem:fsrsm2}.
Since $r$ is bounded, we can move the Weber term to the right-hand side and consider
\begin{align*} v''(r)&-\tfrac{\nu^2-\frac14}{r^2}v(r)
-\mu v(r) 
=r^2v(r)-V(r)v(r). 
\end{align*}
Thus, the dominant contribution comes from a Bessel equation and we construct
a fundamental system by perturbing Hankel functions.
\item Bessel case: $0<r<c\omega^{-\frac12}$, Lemma \ref{lem:fsrsm}.
This is similar to the Hankel case but we replace the Hankel functions by Bessel functions
in order to gain control near $r=0$.
\end{enumerate}
\item Large angular momenta: $\nu\geq \nu_0$.
\begin{enumerate}
\item Weber case: $r\geq 1$, Proposition \ref{prop:fsrbig}. This case is in fact already
handled by the small angular momenta Weber case since it turns out 
that for the fundamental system $\{v_\pm\}$ one has good control for all $\nu\geq 0$
and not just $\nu\leq \nu_0$.
\item Bessel case: $0<r\leq 1$, Lemma \ref{lem:fsrsmnlg}. 
This is reminiscent of the classical asymptotic theory
for Bessel functions. We rewrite Eq.~\eqref{eq:out_vhom} as
\begin{align*} v''(r)&-\tfrac{\nu^2}{r^2}v(r)
-\mu v(r) \\
&=-\tfrac{1}{4r^2}v(r)+r^2v(r)-V(r)v(r)
\end{align*}
and construct a fundamental system by perturbing the asymptotic form of Bessel functions
for large parameters.
\end{enumerate}
\end{enumerate}
The representations of solutions to Eq.~\eqref{eq:out_vhom} in the various regimes
are then used to compute the Wronskian $W(v_0,v_-)$ and to estimate the Green function
for large $\omega$, uniformly in $\nu$.
The bound on the reduced resolvent can then be obtained in a straightforward manner.

\subsection{A remark on notation}
In the technical part we make extensive use of the ``$O_\C$-notation''. In this respect
we would like to reiterate that there is no hidden dependence of any kind on
the relevant parameters $\omega$ and $\nu$.
On the contrary, it is of course precisely the point of our whole construction 
to track the 
dependence
on $\omega$ and $\nu$ explicitly through all computations.
Due to the complexity of the calculations, this is only possible with an economic notation that keeps track of the relevant information
but suppresses all the irrelevant details.
In particular, as is standard in many branches of analysis,
we hardly ever denote absolute constants explicitly.
The dependence on $\omega$ and $\nu$, on the other hand, is always explicitly stated, even if it is not
relevant in the particular context.
For instance, we use the notation $O_\C(\omega^{-\frac12}\nu^0)$ for a complex-valued function
that depends on $\omega$, $\nu$ and which is bounded by $C\omega^{-\frac12}$ for
some absolute constant $C>0$ in the relevant range of $\omega$ and $\nu$ (which is also
stated explicitly).

\section{Construction of fundamental systems}

\subsection{Reduction to normal form}

In order to construct the reduced resolvent $R_{\ell,m}(\lambda)$ we have to solve the
equation
\begin{align}
\label{eq:specrad} u''(r)&+\tfrac{d-1}{r}u'(r)-2 r u'(r)
-\tfrac{\ell(\ell+d-2)}{r^2}u(r) \nonumber \\
&+V(r)u(r)-\lambda u(r)=-f(r). 
\end{align}
Setting $u(r):=r^{-\frac{d-1}{2}}e^{r^2/2}v(r)$ yields the normal form equation
\begin{align}
\label{eq:specnorm}
v''(r)&-r^2 v(r)-\tfrac{(d+2\ell-1)(d+2\ell-3)}{4r^2}v(r)
-(\lambda-d)v(r) \nonumber \\
&+V(r)v(r)=-r^{\frac{d-1}{2}}e^{-r^2/2}f(r).
\end{align}

\noindent Consequently, our first task is to construct a fundamental system for 
\begin{align}
\label{eq:spechom}
v''(r)&-r^2 v(r)-\tfrac{(d+2\ell-1)(d+2\ell-3)}{4r^2}v(r)
-(\lambda-d)v(r) \nonumber \\
&+V(r)v(r)=0.
\end{align}
We set $\nu=\frac{d}{2}+\ell-1$, $\mu=\lambda-d$ and rewrite Eq.~\eqref{eq:spechom}
as 
\begin{align}
\label{eq:spechom2}
v''(r)&-r^2 v(r)-\tfrac{\nu^2}{r^2}v(r)-\mu v(r) =-\tfrac{1}{4r^2}v(r)-V(r)v(r).
\end{align}

\subsection{A fundamental system away from the center}

As suggested by the notation, we intend to treat the right-hand side of Eq.~\eqref{eq:spechom2} 
perturbatively. 
Thus, for the moment we set it to zero and note
that the rescaling $w(y)=v(\mu^\frac12 y)$ with $\mu>0$ yields the equation
\[ w''(y)-\mu^2(1+y^2)w(y)-\tfrac{\nu^2}{y^2}w(y)=0. \]
The Liouville-Green transform $\tilde w(\xi(y))=|\xi'(y)|^\frac12 w(y)$ leads to
\[ \tilde w''(\xi(y))-\frac{\mu^2(1+y^2)+\frac{\nu^2}{y^2}}{\xi'(y)^2}
\tilde w(\xi(y))
-\frac{Q(y)}{\xi'(y)^2}\tilde w(\xi(y))=0
 \]
with
\[ Q(y)=\frac12 \frac{\xi'''(y)}{\xi'(y)}-\frac34 \frac{\xi''(y)^2}{\xi'(y)^2}. \]
Consequently, it is reasonable to look for a diffeomorphism $\xi$ 
that satisfies
\[ \frac{\mu^2(1+y^2)+\tfrac{\nu^2}{y^2}}{\xi'(y)^2}=\mu^2 \]
or
\[ \xi'(y)=\sqrt{1+y^2+\tfrac{\nu^2}{\mu^2 y^2}} \]
and thus,
\[ \xi(y):=\int_{\mu^{-1/2}}^y \sqrt{1+s^2+\tfrac{\nu^2}{\mu^2 s^2}}ds \]
is a possible choice (the lower bound is arbitrary but this choice turns out to be convenient).
A more precise notation would be $\xi(y;\mu,\nu)$ but we refrain from using this in order
to keep the equations shorter. This sloppiness comes at the price of strange-looking identities
like $\xi(\mu^{-\frac12})=0$.
For the Liouville-Green potential we infer
\begin{equation} 
\label{eq:LGQ}
Q(y)=\frac{2y^2-3y^4+6\frac{\alpha^2}{y^2}+18y^2\frac{\alpha^2}{y^2}+\frac{\alpha^4}{y^4}}
{4y^2(1+y^2+\frac{\alpha^2}{y^2})^2},\qquad \alpha:=\tfrac{\nu}{\mu}.
\end{equation}
By construction, the equation
\begin{equation}
\label{eq:walpha} w''(y)-\mu^2(1+y^2)w(y)-\tfrac{\nu^2}{y^2}w(y)+Q(y)w(y)=0 
\end{equation}
transforms into
\[ \tilde w''(\xi(y))-\mu^2 \tilde w(\xi(y))=0 \]
for $\tilde w(\xi(y))=\xi'(y)^\frac12 w(y)$ and thus, Eq.~\eqref{eq:walpha}
has the fundamental system
\[ \xi'(y)^{-\frac12}e^{\pm \mu \xi(y)}. \]
By setting $y=\mu^{-\frac12}r$ and $w(y)=v(\mu^\frac12 y)$, Eq.~\eqref{eq:walpha} transforms into
\begin{equation}
\label{eq:rbighom}
v''(r)-r^2 v(r)-\tfrac{\nu^2}{r^2}v(r)+\mu^{-1}Q(\mu^{-\frac12}r)v(r)
-\mu v(r)=0
\end{equation}
with the fundamental system
\begin{equation}
\label{eq:basicfs}
 \xi'(\mu^{-\frac12}r)^{-\frac12}e^{\pm \mu \xi(\mu^{-1/2}r)}. 
 \end{equation}
This suggests to rewrite Eq.~\eqref{eq:spechom2} as
\begin{align}
\label{eq:rbig}
v''(r)&-r^2 v(r)-\tfrac{\nu^2}{r^2}v(r)+\mu^{-1}Q(\mu^{-\frac12}r)v(r)
-\mu v(r) \nonumber \\
&=\mu^{-1}Q(\mu^{-\frac12}r)v(r)-\tfrac{1}{4r^2}v(r)-V(r)v(r)
\end{align}
and the hope is to treat the right-hand side perturbatively.

\subsubsection{Analysis of $\xi$ and $Q$}
So far we were dealing with $\mu>0$ but we are actually interested in $\mu=b+\I\omega$
for fixed $b\in \R$ and $\omega$ large.
For $\mu>0$ the function $\xi$ can be written as
\begin{align} 
\label{eq:xi}
\xi(\mu^{-\frac12}r)&=\int_{\mu^{-1/2}}^{\mu^{-1/2}r}\sqrt{1+s^2+\tfrac{\nu^2}{\mu^2 s^2}}ds \nonumber \\
&=\mu^{-\frac12}\int_1^r \sqrt{1+\tfrac{s^2}{\mu}+\tfrac{\nu^2}{\mu s^2}}ds
\end{align}
and the last expression makes perfect sense even for $\mu=b+\I\omega$ as a contour integral
of a holomorphic function (the argument of the square root stays in $\C\backslash (-\infty,0]$
for all $s\geq 1$). 
We note that $\sqrt{\cdot}$ always means the principal branch of the complex square root,
holomorphic in $\C \backslash (-\infty,0]$ and explicitly given by
\begin{equation}
\label{eq:sqrt}
 \sqrt z=\frac{1}{\sqrt 2}\sqrt{|z|+\Re z}+\frac{\I\sgn(\Im z)}{\sqrt 2}\sqrt{|z|-\Re z}. 
 \end{equation}
As a direct consequence of the explicit formula \eqref{eq:sqrt} we have 
$\sqrt{z}^2=z$ and $|\sqrt{z}|=\sqrt{|z|}$ for all $z\in \C\backslash (-\infty,0]$.
Furthermore, the formula 
\[ \sqrt{z}\sqrt{w}=\sqrt{zw} \] is valid
at least if $\Re z\geq 0$ and $\Re w>0$.

Based on the explicit expression \eqref{eq:LGQ},
the Liouville-Green potential $\mu^{-1}Q(\mu^{-\frac12}r)$ has a straightforward 
analytic continuation
to $\mu=b+\I\omega$.
As a consequence, \eqref{eq:basicfs} is a fundamental system for Eq.~\eqref{eq:rbighom}
also in the complex case $\mu=b+\I\omega$.
We remark that in general, all functions depend on the parameter $\nu$ and this dependence
is crucial. However,
for the sake of readability we usually suppress it in the notation.

The following bound shows that the right-hand side of Eq.~\eqref{eq:rbig} can
indeed be treated perturbatively.

\begin{lemma}
\label{lem:estQ}
Let $\mu=b+\I\omega$ where $b\in \R$ is fixed. Then we have the bound
\[ |\mu^{-1}Q(\mu^{-\frac12}r)|\lesssim r^{-2} \]
for all $r>0$, $\omega \gg 1$, and $\nu\geq 0$.
\end{lemma}
 
\begin{proof}
We set $\alpha=\frac{\nu}{\mu}$.
For all $\alpha,y\in \C$
we have the bound
\begin{equation}
\label{eq:estQ0}
 |Q(y)|\lesssim \frac{|\frac{\alpha}{y}|^2+|\frac{\alpha}{y}|^4}{|y|^2|1+y^2
+\frac{\alpha^2}{y^2}|^2}
+\frac{|y|^2+|y|^4}{|y|^2|1+y^2+\frac{\alpha^2}{y^2}|^2}. 
\end{equation}
In order to estimate the denominator, we use the bound
\begin{equation}
\label{eq:estx}
 |1+\tfrac{x}{\mu}|^2\geq \tfrac12(1+\tfrac{x^2}{|\mu|^2}),\qquad x\in \R
 \end{equation}
which holds for $\omega\gg 1$ as a consequence of
 \begin{align*}
 |1+\tfrac{x}{\mu}|^2&=1+2x\Re(\mu^{-1})+\tfrac{x^2}{|\mu|^2} \\
 &=1+\frac{2b}{\sqrt{b^2+\omega^2}}\frac{x}{\sqrt{b^2+\omega^2}}+\frac{x^2}{b^2+\omega^2} \\
 &\geq 1-\frac{2b^2}{b^2+\omega^2}+\frac12 \frac{x^2}{b^2+\omega^2}.
 \end{align*}
Thus, Eqs.\ \eqref{eq:estx} and \eqref{eq:estQ0} imply
\[ |Q(\mu^{-\frac12}r)|\lesssim |\mu|r^{-2} \]
which yields the claim.
\end{proof}

The following representation of $\xi$ is crucial and contains in fact all the information
on $\xi$ we are going to use.

\begin{lemma}
\label{lem:xi}
Let $\mu=b+\I\omega$ where $b\in \R$ is fixed. Then we have the representation
\begin{align*} 
\Re [\mu\xi(\mu^{-\frac12}r)]&=\tfrac12 r^2+\tfrac{b}{2}\log 
\langle \mu^{-\frac12}r\rangle
+\varphi(r;\omega,\nu)
\end{align*}
where $\partial_r \varphi(r;\omega,\nu)\geq 0$ for all $r> 0$, $\omega \gg 1$, and $\nu\geq 0$.
\end{lemma}

\begin{proof}
We only prove the case $b\geq 0$.
From Eq.\ \eqref{eq:xi} we obtain
\begin{align*} \partial_r \Re[\mu \xi(\mu^{-\frac12}r)]&=\Re[\mu\partial_r \xi(\mu^{-\frac12}r)] 
=\Re \sqrt{\mu+r^2+\tfrac{\nu^2}{r^2}} \\
&=\tfrac{1}{\sqrt 2}\sqrt{|\mu+r^2+\tfrac{\nu^2}{r^2}|+b+r^2+\tfrac{\nu^2}{r^2}} \\
&\geq \sqrt{b+r^2}
\end{align*}
since
\begin{align*} |\mu+r^2+\tfrac{\nu^2}{r^2}|^2&=|\mu|^2
+2b(r^2+\tfrac{\nu^2}{r^2})+(r^2+\tfrac{\nu^2}{r^2})^2 \\
&\geq  b^2+2b r^2+r^4=(b+r^2)^2.
\end{align*}
Consequently, we find
\begin{align*}
 \partial_r\varphi(r;\omega,\nu)&=
 \partial_r \left \{\Re[\mu\xi(\mu^{-\frac12}r)]-\tfrac12 r^2
 -\tfrac{b}{2}\log 
\langle \mu^{-\frac12}r\rangle\right \} \\
 &\geq \sqrt{b+r^2}-r -\tfrac{b}{2}\frac{|\mu|^{-\frac12}|\mu|^{-\frac12}r}
 {\langle \mu^{-\frac12}r\rangle^2} \\
 &\geq \frac{b}{\sqrt{b+r^2}+r}
 -\frac{\frac{b}{2}|\mu|^{-\frac12}}{\langle \mu^{-\frac12}r\rangle} \\
 &=\frac{b|\mu|^{-\frac12}}{\sqrt{\frac{b}{|\mu|}+\frac{r^2}{|\mu|}}+|\mu|^{-\frac12}r}
 -\frac{\frac{b}{2}|\mu|^{-\frac12}}{\langle \mu^{-\frac12}r\rangle} \\
 &\geq \frac{\frac{b}{2}|\mu|^{-\frac12}}{\langle \mu^{-\frac12}r\rangle}
 -\frac{\frac{b}{2}|\mu|^{-\frac12}}{\langle \mu^{-\frac12}r\rangle}=0.
 \end{align*}
\end{proof}

\subsubsection{A fundamental system for Eq.~\eqref{eq:rbig}}
The functions
\begin{align*} v_0^\pm(r,\omega):&=\tfrac{1}{\sqrt 2} \mu^{-\frac14}\xi'(\mu^{-\frac12}r)^{-\frac12}e^{\pm \mu \xi(\mu^{-1/2}r)},\quad 
\mu=b+\I\omega
\end{align*}
are solutions to Eq.~\eqref{eq:rbighom}.
In order to compute their Wronskian,
we note that for any holomorphic function $f$ and $r\in \R$ we have the chain rule
$\partial_r f(rz)=zf'(rz)$, $z\in \C$.
Consequently, we obtain
\begin{align*} 
\partial_r v_0^\pm(r,\omega)&=\tfrac{1}{\sqrt 2}\mu^{-\frac14}\left [
\partial_r[\xi'(\mu^{-\frac12}r)^{-\frac12}]
\pm \mu^\frac12 \xi'(\mu^{-\frac12}r)^\frac12 \right ]e^{\pm \mu
\xi(\mu^{-1/2}r)}
\end{align*}
which yields
\begin{equation}
W(v_0^-(\cdot,\omega),v_0^+(\cdot,\omega))=1.
\end{equation}
With this information at hand we are ready to construct a fundamental system for
Eq.~\eqref{eq:rbig}.
Note that by Lemma \ref{lem:estQ}, the right-hand side of Eq.~\eqref{eq:rbig} is 
$O_\C(r^{-2}\omega^0\nu^0)v(r)$.

\begin{proposition}\label{prop:fsrbig}
Let $\mu=b+\I\omega$ where $b\in \R$ is fixed.
Then Eq.~\eqref{eq:rbig} has a fundamental system $\{v_-,v_+\}$
of the form
\begin{align*}
 v_\pm(r,\omega)&=v_0^\pm(r,\omega)[1+O_\C(r^{-1}\omega^{-\frac12}\nu^0)] \\
 &=\tfrac{1}{\sqrt 2}\mu^{-\frac14}\xi'(\mu^{-\frac12}r)^{-\frac12}
 e^{\pm \mu \xi(\mu^{-1/2}r)}[1+O_\C(r^{-1}\omega^{-\frac12}\nu^0)] 
\end{align*}
for all $r\geq 1$, $\omega \gg 1$, and $\nu\geq 0$.
\end{proposition}

\begin{proof}
We only treat the case $b\geq 0$.
In order to construct a solution $v_-(r,\omega)$ of Eq.~\eqref{eq:rbig}
that behaves like $v_0^-(r,\omega)$ as $r\to\infty$,
we consider the integral equation
\begin{align*} v_-(r,\omega)=&v_0^-(r,\omega) 
+v_0^+(r,\omega)
\int_r^\infty v_0^-(s,\omega)O_\C(s^{-2}\omega^0\nu^0)v_-(s,\omega)ds \\
&-v_0^-(r,\omega)
\int_r^\infty v_0^+(s,\omega)O_\C(s^{-2}\omega^0\nu^0)v_-(s,\omega)ds.
\end{align*}
We note that $|v_0^\pm(r,\omega)|>0$ for all $r>0$ and 
set $h_-:=\frac{v_-}{v_0^-}$ which yields the Volterra equation
\[
h_-(r,\omega)=1+\int_r^\infty K(r,s,\omega)h_-(s,\omega)ds
\]
where
\[ K(r,s,\omega):=\left [\frac{v_0^+(r,\omega)}{v_0^-(r,\omega)}
v_0^-(s,\omega)^2-v_0^- v_0^+(s,\omega)\right ]O_\C(s^{-2}\omega^0\nu^0). \]
Now we prove pointwise bounds on $K$.
We have
\begin{align*} |v_0^- v_0^+(s,\omega)|&=\tfrac12 |\mu|^{-\frac12}|\xi'(\mu^{-\frac12}s)|^{-1}
=\tfrac12 |\mu|^{-\frac12}|1+\tfrac{1}{\mu}(s^2+\tfrac{\nu^2}{s^2})|^{-\frac12} \\
&\lesssim \omega^{-\frac12}. 
\end{align*}
Furthermore,
\begin{align*} \left |\frac{v_0^+(r,\omega)}{v_0^-(r,\omega)}v_0^-(s,\omega)^2 \right |
&=\tfrac12 |\mu|^{-\frac12}\left |e^{2\mu\xi(\mu^{-1/2}r)}\xi'(\mu^{-\frac12}s)^{-1}e^{-2\mu \xi(\mu^{-1/2}s)} 
\right |  \\
&\lesssim \omega^{-\frac12} e^{-2 \Re[ \mu\xi(\mu^{-1/2}s)-\mu\xi (\mu^{-1/2}r)]} \\
&\lesssim \omega^{-\frac12}\langle \mu^{-\frac12}r\rangle^{b}
\langle \mu^{-\frac12}s\rangle^{-b} e^{-(s^2-r^2)} \\
&\quad \times e^{-2[\varphi(s;\omega,\nu)-\varphi(r;\omega,\nu)]} \\
&\lesssim \omega^{-\frac12}
\end{align*}
for all $0\leq r\leq s$, $\omega \gg 1$, and $\nu\geq 0$ by Lemma \ref{lem:xi}.
As a consequence, we infer the estimate $|K(r,s,\omega)|\lesssim \omega^{-\frac12}s^{-2}$
which implies
\[ \int_1^\infty \sup_{r\in [1,s]}|K(r,s,\omega)|ds\lesssim \omega^{-\frac12}
\int_1^\infty s^{-2}ds\lesssim 1 \]
and the standard result on Volterra equations (see Remark \ref{rem:Volterra} below)
yields the existence of a solution $h_-$
with the bound $|h_-(r,\omega)-1|\lesssim r^{-1}\omega^{-\frac12}$.

In order to construct the solution $v_+$ we note that $|v_-(r,\omega)|>0$ for all $r\geq 0$
and $\omega \gg 1$.
Consequently, a second solution of Eq.~\eqref{eq:rbig} is given by
\[ v_+(r,\omega):=v_-(r,\omega)\left [\frac{v_0^+(1,\omega)}
{v_0^-(1,\omega)}
+\int_1^r v_-(s,\omega)^{-2}ds \right ]. \]
We set $h_+:=\frac{v_+}{v_0^+}$.
The identity 
\[ v_0^-(r,\omega)\int_1^r v_0^-(s,\omega)^{-2}ds=v_0^+(r,\omega)
-\frac{v_0^+(1,\omega)}{v_0^-(1,\omega)}v_0^-(r,\omega), \]
which follows from $W(v_0^-(\cdot,\omega),v_0^+(\cdot,\omega))=1$, then yields
\begin{align*}
h_+(r,\omega)=1&+O_\C(r^{-1}\omega^{-\frac12}\nu^0) \\
&-\frac{v_-(r,\omega)}{v_0^+(r,\omega)}\int_1^r 
v_0^-(s,\omega)^{-2}O_\C(s^{-1}\omega^{-\frac12}\nu^0)ds.
\end{align*}
Now recall that 
\begin{align*} 
v_0^-(s,\omega)^{-2}&=2\mu^\frac12 \xi'(\mu^{-\frac12}s)e^{2\mu \xi(\mu^{-1/2}s)} \\
&=\partial_{s}e^{2\mu \xi(\mu^{-1/2}s)}
\end{align*}
and thus, an integration by parts yields
\begin{align*}
h_+(r,\omega)&=1+O_\C(r^{-1}\omega^{-\frac12}\nu^0) \\
&\quad +e^{-2\mu\xi(\mu^{-1/2}r)}[1+O_\C(r^{-1}\omega^{-\frac12}\nu^0)] \\
&\quad \times \Big [e^{2\mu \xi(\mu^{-1/2}r)}O_\C(r^{-1}\omega^{-\frac12}\nu^0)
-e^{2\mu \xi(\mu^{-1/2})}O_\C(\omega^{-\frac12}\nu^0) \\
&\quad +\int_1^r e^{2\mu\xi(\mu^{-1/2}s)}O_\C(s^{-2}\omega^{-\frac12}\nu^0)ds \Big ] \\
&=1+O_\C(r^{-1}\omega^{-\frac12}\nu^0)
\end{align*}
since
\[ \left |e^{-2 [\mu\xi(\mu^{-1/2}r)-\mu\xi(\mu^{-1/2}s)]} \right |\lesssim e^{-(r^2-s^2)} \]
for all $r\geq s$, $\omega\gg 1$, and $\nu\geq 0$, see Lemma \ref{lem:xi}.
\end{proof}

\begin{remark}
\label{rem:Volterra}
The Volterra equation 
\[ f(x)=g(x)+\int_x^\infty K(x,y)f(y)dy \]
has a solution $f\in L^\infty(a,\infty)$ if $g\in L^\infty(a,\infty)$ and
\[ \int_a^\infty \sup_{x\in (a,y)}|K(x,y)|dy<\infty, \]
see e.g.~\cite{SchSofSta10}.
In addition, $f$ inherits differentiability properties from $g$ and $K$.
For instance, if the kernel $K$ is of the form $K(x,y)=e^{\phi(x)-\phi(y)}\tilde K(x,y)$,
where $\phi: (0,\infty)\to (0,\infty)$ is an orientation-preserving diffeomorphism, 
and the functions $\phi$, $\tilde K$, and $g$ behave like symbols\footnote{
A function $f \in C^k(I)$, $I\subset \R$ an interval, 
is said to behave like a symbol if $|f^{(j)}(x)|\leq C|x|^{\gamma-j}$ for some $\gamma\in \R$,
all $0\leq j\leq k$, and all $x\in I$.
A similar notion applies to functions of several variables.
As a consequence of the chain rule, symbol behavior is preserved 
under the usual algebraic operations.},
then $f$ has the same property. This follows by a simple induction 
from the identity
\begin{align*} 
f(x)&=g(x)+\int_x^\infty e^{\phi(x)-\phi(y)}\tilde K(x,y)f(y)dy \\
&=g(x)  +\int_0^\infty e^{-y'}\tilde K(x,\phi^{-1}(y'+\phi(x))) \\
&\quad \times f(\phi^{-1}(y'+\phi(x)))\frac{dy'}
{\phi'(\phi^{-1}(y'+\phi(x)))}
\end{align*}
which shows that $x$-derivatives hit only terms that behave like symbols.
\end{remark}

\subsection{A fundamental system near the center}
\label{sec:rsm}
Next, we consider Eq.~\eqref{eq:spechom} for $0<r\leq 1$.
In this case we move the term $r^2 v(r)$ to the right-hand side and treat
it perturbatively, i.e., we rewrite Eq.~\eqref{eq:spechom2} as
\begin{equation}
\label{eq:rsm}
v''(r)-\frac{\nu^2-\frac14}{r^2}v(r)-\mu v(r)
=O_\C(\langle r\rangle^2)v(r).
\end{equation}
We consider the ``homogeneous'' version
\begin{equation}
\label{eq:rsmhom}
v''(r)-\frac{\nu^2-\frac14}{r^2}v(r)-\mu v(r)=0
\end{equation}
and rescale by introducing $v(r)=w(\mu^\frac12  r)$ with $\mu>0$.
This rescaling yields the modified Bessel equation
\[ w''(y)-\frac{\nu^2-\frac14}{y^2} w(y)-w(y)=0 \]
where $y=\mu^\frac12 r$. 
A fundamental system for this equation is given by
\[ \sqrt{y}J_\nu(\I y),\quad \sqrt y Y_\nu(\I y) \]
where $J_\nu$ and $Y_\nu$ are the standard Bessel functions, see e.g.~\cite{Olv97,
NIST}.
Consequently, Eq.~\eqref{eq:rsmhom} has the fundamental system
\[ \sqrt r J_\nu(\I  \mu^\frac12 r),\quad \sqrt r Y_\nu(\I  \mu^\frac12 r). \]

\begin{lemma}
\label{lem:fsrsm}
Let $\nu_0>0$ and $\mu=b+\I \omega$ where $b\in \R$ is fixed.
Furthermore, fix $c\geq 1$.
Then Eq.~\eqref{eq:rsm} has a fundamental system $\{v_0,v_1\}$
of the form
\begin{align*}
v_0(r,\omega)&=\sqrt r J_\nu(\I \mu^\frac12 r)[1+O_\C(r^0 \omega^{-\frac12})] \\
v_1(r,\omega)&=\sqrt r [Y_\nu(\I \mu^\frac12 r)+O_\C(\omega^0)J_\nu(\I\mu^\frac12 r)]
[1+O_\C(r^0 \omega^{-\frac12})]
\end{align*}
for all $r\in (0, c\omega^{-\frac12}]$, $\omega \gg c^2$, and $\nu\in [0,\nu_0]$.
\end{lemma}

\begin{proof}
The Bessel function $J_\nu$ is holomorphic in $\C \backslash (-\infty,0]$ and hence,
by analytic continuation,  
\[ \psi_0(r,\omega):=\sqrt r J_\nu(\I \mu^\frac12 r) \]
is a solution to Eq.~\eqref{eq:rsmhom} with $\mu=b+\I \omega$.
Recall that all zeros of $J_\nu$ are real (\cite{Olv97}, p.~245, Theorem 6.2) and therefore,
$|\psi_0(r,\omega)|>0$ for all $r>0$.
Thus, we may set
\[ \psi_1(r,\omega):=-\psi_0(r,\omega)\int_r^{c\omega^{-1/2}}\psi_0(s,\omega)^{-2}ds \]
and this yields another solution to Eq.~\eqref{eq:rsmhom}. 
For the Wronskian of $\psi_0$ and $\psi_1$ we have
$W(\psi_0(\cdot,\omega),\psi_1(\cdot,\omega))=1$ which shows that
$\{\psi_0,\psi_1\}$ is a fundamental system for Eq.~\eqref{eq:rsmhom}.
For notational convenience we set $\tilde \psi_1(r,\omega):=\sqrt{r}Y_\nu(\I \mu^\frac12 r)$.
From $W(J_\nu,Y_\nu)(z)=\frac{2}{\pi z}$ we infer $W(\psi_0(\cdot,\omega),
\tilde \psi_1(\cdot,\omega))=\frac{2}{\pi}$ and thus, $\{\psi_0,\tilde \psi_1\}$ is
another fundamental system for Eq.~\eqref{eq:rsmhom}.
Consequently, we have the connection formula
\[ \psi_1(r,\omega)=\frac{W(\psi_1(\cdot,\omega),\tilde \psi_1(\cdot,\omega))}
{W(\psi_0(\cdot,\omega),\tilde \psi_1(\cdot,\omega))}\psi_0(r,\omega)
+\frac{W(\psi_1(\cdot,\omega),\psi_0(\cdot,\omega))}
{W(\tilde \psi_1(\cdot,\omega), \psi_0(\cdot,\omega))}\tilde \psi_1(r,\omega) \]
and by evaluation at $r=c\omega^{-\frac12}$ we find
\begin{align*} 
W(\psi_1(\cdot,\omega),\tilde \psi_1(\cdot,\omega))&=
-\frac{\tilde \psi_1(c\omega^{-\frac12},\omega)}{\psi_0(c\omega^{-\frac12},\omega)}
=-\frac{Y_\nu(\I c\mu^\frac12 \omega^{-\frac12})}{J_\nu(\I c\mu^\frac12 \omega^{-\frac12})}
=O_\C(\omega^0).
\end{align*}
This yields the representation
\[ \psi_1(r,\omega)=\tfrac{\pi}{2}\sqrt{r}Y_\nu(\I \mu^\frac12 r)
+O_\C(\omega^0)\sqrt{r}J_\nu(\I \mu^\frac12 r). \]

In order to construct $v_0$, 
we have to look for a solution of the integral equation
\begin{align*} 
v_0(r,\omega)=&\psi_0(r,\omega)-\psi_0(r,\omega)\int_0^r \psi_1(s,\omega)O_\C(\langle s\rangle^2)
v_0(s,\omega)ds \\
&+\psi_1(r,\omega)\int_0^r \psi_0(s,\omega)O_\C(\langle s\rangle^2)v_0(s,\omega)ds.
\end{align*}
Consequently, upon setting $v_0=\psi_0 h_0$, 
we obtain the Volterra equation
\begin{equation}
\label{eq:Volterrah0} h_0(r,\omega)=1+\int_0^r K(r,s,\omega)h_0(s,\omega)ds 
\end{equation}
with the kernel
\[ K(r,s,\omega)=\left [\frac{\psi_1(r,\omega)}{\psi_0(r,\omega)}\psi_0(s,\omega)^2
-\psi_0(s,\omega)\psi_1(s,\omega)\right ]O_\C(\langle s\rangle^2). \]
We have $|\psi_0(r,\omega)|\simeq \omega^{\frac{\nu}{2}}r^{\frac12+\nu}$ 
and\footnote{We assume
here for simplicity that $d$ is odd, i.e., $\nu=\frac{d}{2}+\ell-1$ is not an integer. 
In the case $d$ even one may encounter
a logarithmic loss but this does not affect the final result.}
$|\psi_1(r,\omega)|\simeq \omega^{-\frac{\nu}{2}}r^{\frac12-\nu}$
for $0<r\leq c\omega^{-\frac12}$ and $\omega \gg c^2$ and thus,
$|K(r,s,\omega)|\lesssim s$ for $0\leq s\leq r\leq c\omega^{-\frac12}$.
This implies
\[ \int_0^{c\omega^{-1/2}}
\sup_{r\in [s,c\omega^{-1/2}]}|K(r,s,\omega)|ds\lesssim \omega^{-\frac12} \]
and the standard result on Volterra equations yields a solution of Eq.~\eqref{eq:Volterrah0}
with the bound $|h_0(r,\omega)-1|\lesssim \omega^{-\frac12}$.
Since $|v_0(r,\omega)|>0$ for all $r \in (0,c\omega^{-\frac12}]$ provided $\omega$
is sufficiently large, a second solution $v_1$ is given by
\[ v_1(r,\omega)=-\tfrac{2}{\pi}v_0(r,\omega)\int_r^{c\omega^{-1/2}}v_0(s,\omega)^{-2}ds  \]
and it is straightforward to verify that it is indeed of the stated form.
\end{proof}

Unfortunately, we cannot directly glue together the fundamental systems from Proposition
\ref{prop:fsrbig} and Lemma \ref{lem:fsrsm}.
As a consequence, we still require another fundamental system which allows us to bridge
the gap $r\in [c\omega^{-\frac12},1]$.
The latter is obtained by perturbing Hankel functions $H^\pm_\nu=J_\nu\pm \I Y_\nu$.

\begin{lemma}
\label{lem:fsrsm2}
Let $\nu_0>0$ and $\mu=b+\I\omega$ where $b\in \R$ is fixed.
Then there exists a fundamental system $\{\tilde v_-,\tilde v_+\}$ for Eq.~\eqref{eq:rsm}
of the form
\[ \tilde v_\pm(r,\omega)=\sqrt r H_\nu^{\mp}(\I\mu^\frac12 r)[1+O_\C(r^0\omega^{-\frac12})] \]
for all $r\in [c\omega^{-\frac12},1]$, $\omega \gg c^2$, and $\nu\in [0,\nu_0]$,
provided $c\geq 1$ is sufficiently large.
\end{lemma}

\begin{proof}
We set
\[ \psi_\pm(r,\omega):=c_\nu^\pm \sqrt r H_\nu^\mp (\I \mu^\frac12 r),\quad \mu=b+\I\omega \]
where $c_\nu^\pm$ are constants which will be chosen below.
By analytic continuation, $\psi_\pm(\cdot,\omega)$ are solutions to Eq.~\eqref{eq:rsmhom}.
Furthermore, from the standard Hankel asymptotics 
\[ H_\nu^\pm(z)=\sqrt{\frac{2}{\pi z}}e^{\pm \I (z-\frac12 \nu\pi-\frac14 \pi)}[1+O_\C(z^{-1})] \]
we see that $c_\nu^\pm$ can be chosen in such a way that
\[ \psi_\pm(r,\omega)=\tfrac{1}{\sqrt 2}\mu^{-\frac14}e^{\pm \mu^{1/2}r}[1+O_\C(r^{-1}\omega^{-\frac12})] \]
provided $|\mu^\frac12 r|\geq 1$ and $\omega \gg 1$.
It follows that $W(\psi_-(\cdot,\omega),\psi_+(\cdot,\omega))=1$ and thus,
$\{\psi_\pm(\cdot,\omega)\}$ is a fundamental system for Eq.~\eqref{eq:rsmhom}.
Consequently, we intend to construct a solution of the integral equation
\begin{align*} \tilde v_-(r,\omega)=\psi_-(r,\omega)&+\psi_-(r,\omega)\int_r^1 
\psi_+(s,\omega)O_\C(\langle s\rangle^2)
\tilde v_-(s,\omega)ds  \\
& -\psi_+(r,\omega)\int_r^1 \psi_-(s,\omega)O_\C(\langle s\rangle^2)\tilde v_-(s,\omega)ds.
\end{align*}
The functions $\psi_\pm(\cdot,\omega)$ do not have zeros on $[c\omega^{-\frac12},\infty)$
if $c\geq 1$ is sufficiently large and thus, we may set
$\tilde v_-=\psi_- h_-$ and derive the Volterra equation
\[ h_-(r,\omega)=1+\int_r^1 K(r,s,\omega)h_-(s,\omega)ds \]
for the function $h_-$ where
\[ K(r,s,\omega)=\left [\psi_-(s,\omega)\psi_+(s,\omega)
-\frac{\psi_+(r,\omega)}{\psi_-(r,\omega)}\psi_-(s,\omega)^2 \right ]O_\C(\langle s\rangle^2). \]
We derive the bound
\[ |K(r,s,\omega)|\lesssim |\mu|^{-\frac12}+\left |\mu^{-\frac12}e^{-2\mu^{1/2}(s-r)}\right |
\lesssim \omega^{-\frac12} \]
provided $c\omega^{-\frac12}\leq r\leq s\leq 1$ and thus,
\[ \int_{c\omega^{-1/2}}^1 \sup_{r\in [c\omega^{-1/2},s]}|K(r,s,\omega)|ds \lesssim \omega^{-\frac12}.\]
Consequently, a standard Volterra iteration yields a solution $h_-$ with the bound
$|h_-(r,\omega)-1|\lesssim \omega^{-\frac12}$ for $r\in [c\omega^{-\frac12},1]$.
Dividing by $c_\nu^-$ yields the stated form of $\tilde v_-$.
The solution $\tilde v_+$ follows from the reduction ansatz, cf.~the proof of Proposition
\ref{prop:fsrbig}.
\end{proof}

Thanks to the global representation $H^\pm_\nu=J_\nu \pm \I Y_\nu$ it is easy to glue
together the fundamental systems from Lemmas \ref{lem:fsrsm} and \ref{lem:fsrsm2}.

\begin{corollary}
\label{cor:rsm}
Let $\nu_0>0$. Then
we have the representations
\begin{align*}
\tilde v_\pm(r,\omega)&=O_\C(\omega^0)v_0(r,\omega)\mp [\I+O_\C(\omega^{-\frac12})]
v_1(r,\omega) \\
v_0(r,\omega)&=[\tilde \beta_-+O_\C(\omega^{-\frac12})]\tilde v_-(r,\omega)
+[\tilde \beta_++O_\C(\omega^{-\frac12})]\tilde v_+(r,\omega)
\end{align*}
for all $r\in (0,1]$, $\omega \gg 1$, and $\nu\in [0,\nu_0]$ where
$\tilde \beta_\pm \in \C\backslash \{0\}$.
\end{corollary}

\begin{proof}
Since $\{v_0, v_1\}$ is a fundamental system for Eq.~\eqref{eq:rsm} and $\tilde v_-$
is a solution to that equation, there must exist connection coefficients $\alpha_j(\omega)$
such that
\[ \tilde v_-(r,\omega)=\alpha_0(\omega)v_0(r,\omega)+\alpha_1(\omega)v_1(r,\omega). \]
The coefficients are given by the Wronskian expressions
\[ \alpha_0(\omega)=\frac{W(\tilde v_-(\cdot,\omega),v_1(\cdot,\omega))}
{W(v_0(\cdot,\omega),v_1(\cdot,\omega))},\qquad
\alpha_1(\omega)=\frac{W(\tilde v_-(\cdot,\omega),v_0(\cdot,\omega))}
{W(v_1(\cdot,\omega),v_0(\cdot,\omega))}. \]
For the following computations it is useful to recall the formulae 
\begin{align*} W(pf,qg)&=W(p,q)fg+pqW(f,g) \\
W(f\circ \varphi,g\circ \varphi)&=[W(f,g)\circ \varphi]\varphi' .
\end{align*}
We have
\begin{align*} W(v_0(\cdot,\omega),v_1(\cdot,\omega))&=\I\mu^\frac12 rW(J_\nu,Y_\nu)(\I\mu^\frac12 r)
[1+O_\C(\omega^{-\frac12})] \\
&=\tfrac{2}{\pi}+O_\C(\omega^{-\frac12}).
\end{align*}
Furthermore, by evaluating the Wronskians at $r=c\omega^{-\frac12}$ we find
\begin{align*} 
W(\tilde v_-(\cdot,\omega),v_1(\cdot,\omega))&=\I\mu^\frac12 rW(H_\nu^+,
Y_\nu+O_\C(\omega^0)J_\nu)(\I\mu^\frac12 r)+O_\C(\omega^{-\frac12}) \\
&=O_\C(\omega^0) 
\end{align*}
and
\begin{align*} 
W(\tilde v_-(\cdot,\omega),v_0(\cdot,\omega))&=\I\mu^\frac12 rW(H_\nu^+,J_\nu)(\I\mu^\frac12 r)
+O_\C(\omega^{-\frac12}) \\
&=\I\mu^\frac12 rW(J_\nu+\I Y_\nu,J_\nu)(\I\mu^\frac12 r)
+O_\C(\omega^{-\frac12}) \\
&=-\tfrac{2\I}{\pi}+O_\C(\omega^{-\frac12}).
\end{align*}
Consequently, we infer $\alpha_0(\omega)=O_\C(\omega^0)$ and $\alpha_1(\omega)
=\I+O_\C(\omega^{-\frac12})$ as claimed.
The proof for $\tilde v_+$ is analogous.
For the representation of $v_0$ we use
\[ v_0(r,\omega)=\frac{W(v_0(\cdot,\omega),\tilde v_+(\cdot,\omega))}
{W(\tilde v_-(\cdot,\omega),\tilde v_+(\cdot,\omega))}\tilde v_-(r,\omega)
+\frac{W(v_0(\cdot,\omega),\tilde v_-(\cdot,\omega))}{W(\tilde v_+(\cdot,\omega),
\tilde v_-(\cdot,\omega))}\tilde v_+(r,\omega). \]
\end{proof}

\subsection{A global fundamental system for small angular momenta}
In order to obtain a global fundamental system, it suffices to derive
a representation of $v_-$ in terms of the basis $\{\tilde v_-,\tilde v_+\}$.

\begin{lemma}
\label{lem:fsrbig2}
Let $\nu_0>0$ and $\mu=b+\I\omega$ where $b\in \R$ is fixed.
Then we have the representation
\[ v_-(r,\omega)=\tilde \alpha_-e^{\mu^{1/2}}[1+O_\C(\omega^{-\frac12})]\tilde v_-(r,\omega)
+e^{-\mu^{1/2}}O_\C(\omega^{-\frac12})\tilde v_+(r,\omega) \]
for all $r\geq c\omega^{-\frac12}$, $\omega \gg c^2$, and $\nu\in [0,\nu_0]$, where 
$\tilde \alpha_-\in \C\backslash \{0\}$ and $c\geq 1$ is sufficiently large.
\end{lemma}

\begin{proof}
We have 
\[ v_-(r,\omega)=\tilde \alpha_-(\omega)\tilde v_-(r,\omega)+\tilde \alpha_+(\omega)\tilde v_+(r,\omega) \]
where 
\[ \tilde \alpha_-(\omega)=\frac{W(v_-(\cdot,\omega),\tilde v_+(\cdot,\omega))}{W(\tilde v_-(\cdot,\omega),
\tilde v_+(\cdot,\omega))},\qquad
\tilde \alpha_+(\omega)=\frac{W(v_-(\cdot,\omega),\tilde v_-(\cdot,\omega))}{W(\tilde v_+(\cdot,\omega),
\tilde v_-(\cdot,\omega))}. \]
We will calculate these Wronskians by evaluation at $r=1$.
From Proposition \ref{prop:fsrbig} and $\xi(\mu^{-\frac12})=0$, see Eq.\ \eqref{eq:xi}, we infer
\begin{align*} 
v_-(1,\omega)&=\tfrac{1}{\sqrt 2}\mu^{-\frac14}[1+O_\C(\omega^{-\frac12})] \\
v_-'(1,\omega)&=-\tfrac{1}{\sqrt 2}\mu^\frac14 [1+O_\C(\omega^{-\frac12})]. 
\end{align*}
Furthermore, the Hankel asymptotics imply
\begin{align*} 
\tilde v_\pm(1,\omega)&=c_\nu^\pm \mu^{-\frac14}e^{\pm \mu^{1/2}}[1+O_\C(\omega^{-\frac12})] \\
\tilde v_\pm'(1,\omega)&=\pm c_\nu^\pm \mu^{\frac14}e^{\pm \mu^{1/2}}[1+O_\C(\omega^{-\frac12})]
\end{align*}
for suitable constants $c_\nu^\pm\in \C\backslash \{0\}$.
This yields 
\[ W(\tilde v_-(\cdot,\omega),\tilde v_+(\cdot,\omega))=2c_\nu^-c_\nu^++O_\C(\omega^{-\frac12}) \]
and
\begin{align*}
W(v_-(\cdot,\omega),\tilde v_+(\cdot,\omega))&=\sqrt 2 c_\nu^+e^{\mu^{1/2}}[1+O_\C(\omega^{-\frac12})] \\
W(v_-(\cdot,\omega),\tilde v_-(\cdot,\omega))&=e^{-\mu^{1/2}}O_\C(\omega^{-\frac12})
\end{align*}
which implies the claim.
\end{proof}

It is now a simple matter to obtain global representations for $v_0$ and $v_-$.
\begin{lemma}
\label{lem:fs}
Let $\nu_0>0$ and $\mu=b+\I\omega$ where $b\in \R$ is fixed. 
Then we have the representations
\begin{align*} v_-(r,\omega)&=e^{\mu^{1/2}}O_\C(\omega^0)v_0(r,\omega) 
+ \alpha_1 e^{\mu^{1/2}}[1+O_\C(\omega^{-\frac12})]
v_1(r,\omega) \\
v_0(r,\omega)&=e^{\mu^{1/2}}O_\C(\omega^{-\frac12})v_-(r,\omega) 
+\alpha_+ e^{\mu^{1/2}}[1+O_\C(\omega^{-\frac12})]v_+(r,\omega)
\end{align*}
for all $r>0$, $\omega \gg 1$, and $\nu\in [0,\nu_0]$, where $\alpha_1, \alpha_+ \in \C\backslash\{0\}$.
\end{lemma}

\begin{proof}
The claimed representation of $v_-$ is a consequence of Corollary \ref{cor:rsm} and Lemma
\ref{lem:fsrbig2}.
For the representation of $v_0$ we note that
\[ \tilde v_\pm(r,\omega)=\frac{W(\tilde v_\pm(\cdot,\omega),v_+(\cdot,\omega))}
{W(v_-(\cdot,\omega),v_+(\cdot,\omega))}v_-(r,\omega)
+\frac{W(\tilde v_\pm(\cdot,\omega),v_-(\cdot,\omega))}
{W(v_+(\cdot,\omega),v_-(\cdot,\omega))}v_+(r,\omega)
\]
and from Proposition \ref{prop:fsrbig} we infer
\[ W(v_-(\cdot,\omega),v_+(\cdot,\omega))=1+O_\C(\omega^{-\frac12}). \]
Furthermore, as in the proof of Lemma \ref{lem:fsrbig2} we obtain
\begin{align*}
W(v_-(\cdot,\omega),\tilde v_+(\cdot,\omega))&=\sqrt 2 c_\nu^+ e^{\mu^{1/2}}[1+O_\C(\omega^{-\frac12})] \\
W(v_-(\cdot,\omega),\tilde v_-(\cdot,\omega))&=e^{-\mu^{1/2}}O_\C(\omega^{-\frac12}) \\
W(\tilde v_-(\cdot,\omega),v_+(\cdot,\omega))&=\sqrt 2 c_\nu^-e^{-\mu^{1/2}}[1+O_\C(\omega^{-\frac12})] \\
W(\tilde v_+(\cdot,\omega), v_+(\cdot,\omega))&=e^{\mu^{1/2}}O_\C(\omega^{-\frac12})
\end{align*}
and the claim follows from Corollary \ref{cor:rsm}.
\end{proof}

\subsection{A fundamental system near the center for large angular momenta}
The fundamental systems constructed in Section \ref{sec:rsm} are not useful as $\nu\to\infty$
since the error terms are not controlled in this limit.
Thus, we need yet another construction which covers the case of large angular momenta.
We rewrite Eq.~\eqref{eq:spechom2} as\footnote{It is a well known ``trick'' from the asymptotic
theory of Bessel functions to leave the singular term $-\frac{1}{4r^2}v(r)$ on the right-hand side,
see \cite{Olv97}.}
\begin{equation}
\label{eq:rsmnlg1}
v''(r)-\tfrac{\nu^2}{r^2}v(r)-\mu v(r)
=-\tfrac{1}{4r^2}v(r)+O_\C(\langle r\rangle^2)v(r)
\end{equation}
and consider the ``homogeneous'' version
\begin{equation}
\label{eq:rsmnlghom1}
v''(r)-\tfrac{\nu^2}{r^2}v(r)-\mu v(r)=0.
\end{equation}
We rescale by introducing $v(r)=w(\nu^{-1}\mu^\frac12  r)$ with $\mu>0$, $\nu\geq 1$,
which yields
\[ w''(y)-\nu^2(1+\tfrac{1}{y^2}) w(y)=0 \]
where $y=\nu^{-1}\mu^\frac12 r$. 
The Liouville-Green transform $\tilde w(\zeta(y))=|\zeta'(y)|^\frac12 w(y)$ leads to
\[ \tilde w''(\zeta(y))-\nu^2 \frac{1+\frac{1}{y^2}}{\zeta'(y)^2}\tilde w(\zeta(y))
-\frac{q(y)}{\zeta'(y)^2}\tilde w(\zeta(y))=0 \]
with
\[ q(y)=\frac12 \frac{\zeta'''(y)}{\zeta'(y)}-\frac34 \frac{\zeta''(y)^2}{\zeta'(y)^2}. \]
Consequently, we would like to have
\[ \zeta'(y)=\sqrt{1+\tfrac{1}{y^2}} \]
and
\begin{equation}
\label{def:zeta} \zeta(y)=\sqrt{1+y^2}+\log\frac{y}{1+\sqrt{1+y^2}}
+\gamma 
\end{equation}
does the job, where $\gamma\in \C$ is a free constant which we will choose in a moment.
With this choice of $\zeta$ we obtain
\[ q(y)=\frac{1+6y^2}{4y^2(1+y^2)^2} \]
and the equation
\[ w''(y)-\nu^2(1+\tfrac{1}{y^2})w(y)+q(y)w(y)=0 \]
has the fundamental system $\zeta'(y)^{-\frac12}e^{\pm \nu\zeta(y)}$.
Consequently, the equation
\begin{equation}
\label{eq:rsmnlghom}
 v''(r)-\tfrac{\nu^2}{r^2}v(r)+\alpha^2 q(\alpha r)v(r)-\mu v(r)=0,
\qquad \alpha=\nu^{-1}\mu^{\frac12} 
\end{equation}
has the fundamental system $\zeta'(\alpha r)^{-\frac12}e^{\pm \nu\zeta(\alpha r)}$
and by choosing $\gamma$ in Eq.~\eqref{def:zeta} accordingly, we may normalize
such that\footnote{Again, one should write $\zeta(\alpha r;\alpha)$ but for
brevity we use the sloppier $\zeta(\alpha r)$.} $\zeta(\alpha)=0$.
This suggests to rewrite Eq.~\eqref{eq:rsmnlg1} as
\begin{align}
\label{eq:rsmnlg}
v''(r)&-\tfrac{\nu^2}{r^2}v(r)+\alpha^2 q(\alpha r)v(r)-\mu v(r)
=\alpha^2 \tilde q(\alpha r)v(r)+O_\C(\langle r\rangle^2)v(r)
\end{align}
with
\begin{align*} 
\alpha^2 \tilde q(\alpha r)&=\alpha^2 q(\alpha r)-\tfrac{1}{4r^2} 
=\alpha^2 \frac{4-\alpha^2 r^2}{4(1+\alpha^2 r^2)^2}.
\end{align*}
We emphasize that the function $\alpha^2 \tilde q(\alpha r)$ is regular at $r=0$ which
is crucial for the following.
This is the reason why one has to leave the term $-\frac{1}{4r^2}v(r)$ on the right-hand side
of Eq.~\eqref{eq:rsmnlg1}.

\subsubsection{Analysis of $\zeta$ and $\tilde q$}
As before, we need the analytic continuations of $\zeta(\alpha r)$ and $\alpha^2 \tilde q(\alpha r)$
for $\alpha=\nu^{-1}\mu^\frac12$ with $\mu=b+\I\omega$.
The analytic continuation of $\tilde q$ is manifest since it is a rational function.
Furthermore, the arguments of the square root and the logarithm in $\zeta(\alpha r)$
stay in $\C\backslash (-\infty,0]$ for all $r\geq 0$, $\omega \gg 1$, and $\nu\geq 1$.
Consequently, the desired analytic continuation is obtained by using principal branches.

\begin{lemma}
\label{lem:tildeq}
Let $\mu=b+\I\omega$ and $\alpha=\nu^{-1}\mu^\frac12$ where $b\in \R$ is fixed.
Then we have the bound
\[ |\alpha^2 \tilde q(\alpha r)|\lesssim |\alpha|^2 \langle \alpha r\rangle^{-2} \]
for all $r\geq 0$, $\omega \gg 1$, and $\nu\geq 1$.
\end{lemma}

\begin{proof}
The statement follows from 
\begin{align*} 
|1+\mu x|^2&=1+2\Re \mu x+|\mu|^2 x^2=1+\tfrac{2b}{|\mu|}|\mu|x+|\mu|^2 x^2 \\
&\geq 1-\frac{2b^2}{b^2+\omega^2}+\tfrac12 |\mu|^2 x^2 \\
&\geq \tfrac12 (1+|\mu|^2 x^2)
\end{align*}
which is true for all $x\in \R$ and $\omega \gg 1$.
\end{proof}

The only information on $\zeta$ we are going to use is the following monotonicity property.

\begin{lemma}
\label{lem:zeta}
Let $\mu=b+\I\omega$ and $\alpha=\nu^{-1}\mu^\frac12$ where $b\geq 0$ is fixed. Then we have
\[ \partial_r \Re \zeta(\alpha r)\geq 0 \]
for all $r>0$, $\omega \gg 1$, and $\nu\geq 1$.
\end{lemma}

\begin{proof}
We have
\begin{align*} 
\Re \zeta(\alpha r)&=\Re \sqrt{1+\alpha^2 r^2}+\Re \log \frac{\alpha r}{1+\sqrt{1+\alpha^2 r^2}}
+c_\alpha \\
&=\tfrac{1}{\sqrt 2}\sqrt{|1+\alpha^2 r^2|+1+\nu^{-2}b r^2}+\Re \log \frac{\alpha r}{1+\sqrt{1+\alpha^2 r^2}}+c_\alpha
\end{align*}
where $c_\alpha$ is independent of $r$.
Since $|1+\alpha^2 r^2|^2=1+2\nu^{-2}br^2+|\alpha|^4 r^4$ and $b\geq 0$, it is evident
that the square root is monotonically increasing.
Thus, it suffices to consider the logarithm.
We have
\begin{align*}
  \partial_r \Re \log \frac{\alpha r}{1+\sqrt{1+\alpha^2 r^2}}&=
 \Re \partial_r \log \frac{\alpha r}{1+\sqrt{1+\alpha^2 r^2}}  \\
 &=\Re \frac{1}{r\sqrt{1+\alpha^2 r^2}} \\
 &\geq 0.
  \end{align*}
\end{proof}

\subsubsection{A fundamental system for Eq.~\eqref{eq:rsmnlg}}
The homogeneous equation \eqref{eq:rsmnlghom} has the fundamental system
\[ \tfrac{1}{\sqrt 2}\mu^{-\frac14}\zeta'(\alpha r)^{-\frac12}e^{\pm \nu\zeta(\alpha r)},
\qquad \alpha=\nu^{-1}\mu^\frac12. \]
Now we construct a perturbative fundamental system for Eq.~\eqref{eq:rsmnlg}.

\begin{lemma}
\label{lem:fsrsmnlg}
Let $\mu=b+\I\omega$ and $\alpha=\nu^{-1}\mu^\frac12$ where $b\geq 0$ is fixed.
Then Eq.~\eqref{eq:rsmnlg} has a fundamental system $\{\hat v_0,\hat v_1\}$ of the form
\begin{align*}
\hat v_0(r,\omega)&=\tfrac{1}{\sqrt 2}\mu^{-\frac14}\zeta'(\alpha r)^{-\frac12}e^{\nu \zeta(\alpha r)}
[1+O_\C(r^0 \omega^{-\frac12}+\nu^{-1})] \\
\hat v_1(r,\omega)&=\tfrac{1}{\sqrt 2}\mu^{-\frac14}\zeta'(\alpha r)^{-\frac12}e^{-\nu \zeta(\alpha r)}
[1+O_\C(r^0 \omega^{-\frac12}+\nu^{-1})] 
\end{align*}
for all $r\in (0,1]$, $\omega \gg 1$, and $\nu\geq 1$.
\end{lemma}

\begin{proof}
We set 
\[ \psi_\pm(r,\omega):=\tfrac{1}{\sqrt 2}\mu^{-\frac14}\zeta'(\alpha r)^{-\frac12}e^{\pm \nu\zeta(\alpha r)} \]
and note that
\[ \partial_r \psi_\pm(r,\omega)=\tfrac{1}{\sqrt 2}\mu^{-\frac14}
\left [\partial_r [\zeta'(\alpha r)^{-\frac12}]\pm \nu\alpha 
\zeta'(\alpha r)^\frac12 \right ]e^{\pm \nu\zeta(\alpha r)} \]
which yields $W(\psi_-(\cdot,\omega),\psi_+(\cdot,\omega))=1$ since $\nu\alpha=\mu^\frac12$.
Consequently, our goal is to solve the integral equation
\begin{align*} 
\hat v_0(r,\omega)&=\psi_+(r,\omega)\\
&\quad +\psi_+(r,\omega)\int_0^r \psi_-(s,\omega)
[\alpha^2 \tilde q(\alpha s)+O_\C(\langle s\rangle^2)]\hat v_0(s,\omega)ds  \\
&\quad -\psi_-(r,\omega)\int_0^r \psi_+(s,\omega)
[\alpha^2 \tilde q(\alpha s)+O_\C(\langle s\rangle^2)]\hat v_0(s,\omega)ds. 
\end{align*}
The function $\psi_+(\cdot,\omega)$ does not have zeros on $(0,\infty)$ and hence, we may
set $h_0(r,\omega):=\frac{\hat v_0(r,\omega)}{\psi_+(r,\omega)}$ which leads to the Volterra
equation
\[ h_0(r,\omega)=1+\int_0^r K(r,s,\omega)h_0(s,\omega)ds \]
for $h_0$ with the kernel
\begin{align*} 
K(r,s,\omega)&=\left [\psi_-(s,\omega)\psi_+(s,\omega)
-\frac{\psi_-(r,\omega)}{\psi_+(r,\omega)}\psi_+(s,\omega)^2 \right ] \\
&\quad \times [\alpha^2 \tilde q(\alpha s)+O_\C(\langle s\rangle^2)]. 
\end{align*}
Now we use Lemma \ref{lem:zeta} and 
$|\zeta(\alpha r)|^{-\frac12}\lesssim 1$
to obtain
\begin{align*}
|K(r,s,\omega)|&\lesssim \omega^{-\frac12}
\left [1+e^{-2\nu[\Re \zeta(\alpha r)-\Re \zeta(\alpha s)]}\right ][|\alpha^2 \tilde q(\alpha s)|
+\langle s\rangle^{-2}] \\
&\lesssim \omega^{-\frac12}|\alpha|^2 \langle \alpha s\rangle^{-2}+\omega^{-\frac12}
\langle s\rangle^{-2}
\end{align*}
for all $0<s\leq r$
where the last estimate follows from Lemma \ref{lem:tildeq}.
Consequently, we infer
\begin{align*}
\int_0^1 \sup_{r\in (s,1)}|K(r,s,\omega)|ds&\lesssim \omega^{-\frac12}
|\alpha|^2 \int_0^1 \langle \alpha s\rangle^{-2}ds+\omega^{-\frac12}\int_0^1 \langle s\rangle^{-2}ds \\
&\lesssim \omega^{-\frac12}|\alpha| \int_0^{|\alpha|}\langle s\rangle^{-2}ds
+\omega^{-\frac12} \\
&\lesssim \nu^{-1}+\omega^{-\frac12}
\end{align*}
since $\omega^{-\frac12}|\alpha|\simeq \nu^{-1}$ and a Volterra iteration yields the existence of $h_0$ with 
$|h_0(r,\omega)|\lesssim 1$ for all $r\in [0,1]$, $\omega \gg 1$, and $\nu\geq 1$.
By re-inserting this estimate into the Volterra equation for $h_0$, we find the desired bound
$|h_0(r,\omega)-1|\lesssim \nu^{-1}+\omega^{-\frac12}$.
The solution $\hat v_1$ is constructed by using the reduction ansatz.
\end{proof}

\subsection{A global fundamental system for large angular momenta}
Now we glue together the fundamental systems from Proposition \ref{prop:fsrbig}
and Lemma \ref{lem:fsrsmnlg} in order to obtain a global fundamental system
for large $\nu$.
This time there is no need for an intermediate regime since the corresponding Wronskians
can be evaluated at $r=1$.

\begin{lemma}
\label{lem:fsnlg}
Let $\mu=b+\I\omega$ where $b\geq 0$ is fixed. Then we have the representations
\begin{align*}
v_-(r,\omega)&=O_\C(\omega^{-\frac12}+\nu^{-1})\hat v_0(r,\omega) 
+[1+O_\C(\omega^{-\frac12}+\nu^{-1})]\hat v_1(r,\omega) \\
\hat v_0(r,\omega)&= O_\C(\omega^{-\frac12}+\nu^{-1})v_-(r,\omega) 
+[1+O_\C(\omega^{-\frac12}+\nu^{-1})]v_+(r,\omega) 
\end{align*}
for all $r>0$, $\omega \gg 1$, and $\nu\geq 1$.
\end{lemma}

\begin{proof}
We have
\[ v_-(r,\omega)=\frac{W(v_-(\cdot,\omega),\hat v_1(\cdot,\omega))}{W(\hat v_0(\cdot,\omega),
\hat v_1(\cdot,\omega))}\hat v_0(r,\omega)
+\frac{W(v_-(\cdot,\omega),\hat v_0(\cdot,\omega))}{W(\hat v_1(\cdot,\omega),\hat v_0(\cdot,\omega))}
\hat v_1(r,\omega)
\]
and from Lemma \ref{lem:fsrsmnlg} we obtain
\[ W(\hat v_0(\cdot,\omega),\hat v_1(\cdot,\omega))=-1+O_\C(\omega^{-\frac12}+\nu^{-1}). \]
Furthermore, since $\zeta(\alpha)=0$ with $\alpha=\nu^{-1}\mu^\frac12$, we obtain from
Lemma \ref{lem:fsrsmnlg} the expressions
\begin{align*}
\hat v_0(1,\omega)&=\tfrac{1}{\sqrt 2}\mu^{-\frac14}\zeta'(\alpha)^{-\frac12}
[1+O_\C(\omega^{-\frac12}+\nu^{-1})] \\
\hat v_0'(1,\omega)&=\tfrac{1}{\sqrt 2}\mu^{\frac14}\zeta'(\alpha)^\frac12
[1+O_\C(\omega^{-\frac12}+\nu^{-1})] \\
\hat v_1(1,\omega)&=\tfrac{1}{\sqrt 2}\mu^{-\frac14}\zeta'(\alpha)^{-\frac12}
[1+O_\C(\omega^{-\frac12}+\nu^{-1})] \\
\hat v_1'(1,\omega)&=-\tfrac{1}{\sqrt 2}\mu^{\frac14}\zeta'(\alpha)^\frac12
[1+O_\C(\omega^{-\frac12}+\nu^{-1})]
\end{align*}
and from Proposition \ref{prop:fsrbig} we infer
\begin{align*}
v_\pm(1,\omega)&=\tfrac{1}{\sqrt 2}\mu^{-\frac14}\xi'(\mu^{-\frac12})^{-\frac12}
[1+O_\C(\omega^{-\frac12}\nu^0)] \\
v_\pm'(1,\omega)&=\pm \tfrac{1}{\sqrt 2}\mu^\frac14 \xi'(\mu^{-\frac12})^\frac12
[1+O_\C(\omega^{-\frac12}\nu^0)].
\end{align*}
Now observe that
\begin{align*} \frac{\xi'(\mu^{-\frac12})^2}{\zeta'(\alpha)^2}&=
\frac{1+\frac{1}{\mu}+\frac{\nu^2}{\mu}}
{1+\frac{\nu^2}{\mu}}=1+\frac{1}{\mu+\nu^2}  \\
&=1+O_\C(\omega^{-1}\nu^{-2})
\end{align*}
and we infer
\begin{align*}
W(v_-(\cdot,\omega),\hat v_1(\cdot,\omega))&=O_\C(\omega^{-\frac12}+\nu^{-1}) \\
W(v_-(\cdot,\omega),\hat v_0(\cdot,\omega))&=1+O_\C(\omega^{-\frac12}+\nu^{-1})
\end{align*}
which yields the claimed representation for $v_-$. The proof for $\hat v_0$ is analogous.
\end{proof}

\section{The reduced resolvent}
\noindent We are now in a position to construct the reduced resolvent 
$R_{\ell,m}(\lambda)$, i.e., the solution
operator to Eq.~\eqref{eq:specrad}.
\subsection{Small angular momenta} 
We start with the solution operator to Eq.~\eqref{eq:specnorm} in the case $\ell\leq \ell_0$
for some fixed $\ell_0>0$.
In this case the functions $v_-$ and $v_0$ from Proposition \ref{prop:fsrbig}
and Lemma \ref{lem:fsrsm}, respectively, are relevant.
By Lemma \ref{lem:fs} we have
\begin{align*} W(\omega):&=W(v_-(\cdot,\omega),v_0(\cdot,\omega)) \\
&=\alpha_+e^{\mu^{1/2}}[1 + O_\C(\omega^{-\frac12})]W(v_-(\cdot,\omega),v_+(\cdot,\omega)) \\
&=\alpha_+ e^{\mu^{1/2}}[1+O_\C(\omega^{-\frac12})] 
\end{align*}
with $\alpha_+ \in \C\backslash \{0\}$ and $\mu=b+\I\omega$, $b\in \R$ fixed (the last equality follows from the representation in
Proposition \ref{prop:fsrbig}). 
In particular, this implies that $\{v_0,v_-\}$ is a fundamental system for the homogeneous version
of Eq.~\eqref{eq:specnorm} (provided $\omega$ is sufficiently large).
Furthermore, $v_0(r,\omega)$ and $v_-(r,\omega)$ are recessive as $r\to 0+$ 
and $r\to \infty$, respectively. 
Thus, the variation of constants formula yields a solution $v$ of
Eq.~\eqref{eq:specnorm} given by
\begin{align*} 
v(r)&=\frac{v_0(r,\omega)}{W(\omega)}\int_r^\infty v_-(s,\omega)
s^{\frac{d-1}{2}}e^{-s^2/2}f(s)ds \\
&\quad +\frac{v_-(r,\omega)}{W(\omega)}\int_0^r v_0(s,\omega)
s^{\frac{d-1}{2}}e^{-s^2/2}f(s)ds.
\end{align*}
We define an operator $\tilde R_\ell(\lambda)$ by
\[ \tilde R_\ell(\lambda)\tilde f(r)
:=\int_0^\infty \tilde G_\ell(r,s,\omega)\tilde f(s)ds \]
where $\lambda=d+b+\I\omega$ and
\begin{equation}
\label{def:G} \tilde G_\ell(r,s,\omega):=\frac{e^{\frac{1}{2}(r^2-s^2)}}{W(\omega)}\left 
\{\begin{array}{l}
v_0(r,\omega)v_-(s,\omega)\quad r\leq s \\
v_-(r,\omega)v_0(s,\omega)\quad r\geq s \end{array} \right . .
\end{equation}
With this notation it follows that 
\[ v(r)=e^{-\frac{1}{2}r^2}\tilde R_\ell(\lambda)\left (|\cdot|^{\frac{d-1}{2}}f\right )(r) \]
and thus, the corresponding solution $u$ to Eq.~\eqref{eq:specrad} is given by
\[ u(r)=R_{\ell,m}(\lambda)f(r)=r^{-\frac{d-1}{2}}e^{\frac{1}{2}r^2}v(r)
=r^{-\frac{d-1}{2}}
\tilde R_\ell(\lambda)\left (|\cdot|^\frac{d-1}{2}f \right)(r). \]
From this equation it also follows that $R_{\ell,m}$ is in fact independent of $m$.
Our goal is to show that
\[ \|R_{\ell,m}(\lambda)f\|_{L^2_\mathrm{rad}(\R^d)}\leq C \|f\|_{L^2_{\mathrm{rad}}(\R^d)},
\qquad \lambda=d+b+\I\omega \]
for all $\omega\gg 1$ and $\ell\leq \ell_0$.
By the above, this is equivalent to the bound
\begin{equation}
\label{eq:bdres}
\|\tilde R_\ell(\lambda)\tilde f\|_{L^2(\R_+)}\leq C \|\tilde f\|_{L^2(\R_+)}
\end{equation}
and a proof of \eqref{eq:bdres} is the goal of this section.

\subsubsection{Kernel bounds}

The desired $L^2$-boundedness of $\tilde R_\ell(\lambda)$ will be a consequence of the
following estimate.

\begin{lemma}
\label{lem:estG}
Let $\ell_0>0$.
Then the kernel $\tilde G_\ell$ defined in \eqref{def:G} satisfies the bound
\[ |\tilde G_\ell(r,s,\omega)|\lesssim \omega^{-\frac12}\left \{\begin{array}{l} 
\langle \omega^{-\frac12}r\rangle^{-\frac12+\frac{b}{2}}
\langle \omega^{-\frac12}s\rangle^{-\frac12-\frac{b}{2}}\quad r\leq s \\
\langle \omega^{-\frac12}r\rangle^{-\frac12-\frac{b}{2}}
\langle \omega^{-\frac12}s\rangle^{-\frac12+\frac{b}{2}} \quad r\geq s \end{array}
\right. \]
for all $r,s>0$, $\omega \gg 1$, and $\ell\leq \ell_0$.
\end{lemma}

\begin{proof}
By symmetry it suffices to consider $r\geq s$. 
Furthermore, we make frequent use of the estimate
\[ |\xi'(\mu^{-\frac12}r)|^{-1}\lesssim \langle \omega^{-\frac12}r\rangle^{-1} \]
which is a consequence of Eq.\ \eqref{eq:estx}.
We distinguish
different cases.
In the following the constant $c\geq 1$ is assumed to be so large that 
Lemma \ref{lem:fsrsm2} holds.
\begin{enumerate}
\item Bessel-Bessel: $0<s\leq r\leq c\omega^{-\frac12}$.
We use Lemma \ref{lem:fs} to obtain the bound
\[ |v_-(r,\omega)|\lesssim e^{\Re(\mu^{1/2})}\left [|v_0(r,\omega)|+|v_1(r,\omega)| \right ]\]
and Lemma \ref{lem:fsrsm} yields, with $\nu=\frac{d}{2}+\ell-1$,
\begin{align*} 
|v_0(s,\omega)v_-(r,\omega)|&\lesssim e^{\Re(\mu^{1/2})}s^\frac12 |\mu^{\frac12} s|^\nu r^\frac12 \left (
|\mu^{\frac12}r|^\nu+|\mu^{\frac12}r|^{-\nu}\right ) \\
&\lesssim e^{\Re(\mu^{1/2})}\omega^{-\frac12} \left (|\omega^\frac12 r|^{2\nu}+1\right ) \\
&\lesssim e^{\Re(\mu^{1/2})}\omega^{-\frac12}.
\end{align*}
This implies
$|\tilde G_\ell(r,s,\omega)|\lesssim \omega^{-\frac12}$.

\item Bessel-Hankel: $0<s\leq c\omega^{-\frac12}\leq r\leq 1$.
From Lemmas \ref{lem:fsrbig2}, \ref{lem:fsrsm2} and the Hankel asymptotics we infer
\begin{align*} |v_-(r,\omega)|&\lesssim e^{\Re(\mu^{1/2})}\omega^{-\frac14}e^{-\Re(\mu^{1/2})r} \\
&\quad +\omega^{-\frac34}e^{-\Re(\mu^{1/2})}e^{\Re(\mu^{1/2})r} \\
&\lesssim \omega^{-\frac14}e^{\Re(\mu^{1/2})}
\end{align*}
and thus,
\begin{align*}
|v_0(s,\omega)v_-(r,\omega)|\lesssim \omega^{-\frac12}e^{\Re(\mu^{1/2})}
\end{align*}
which yields the desired $|\tilde G_\ell(r,s,\omega)|\lesssim \omega^{-\frac12}$.

\item Bessel-Weber: $0<s\leq c\omega^{-\frac12}\leq 1\leq r$.
Here we use Proposition \ref{prop:fsrbig} and Lemma \ref{lem:xi} to obtain
\begin{align*} 
|v_-(r,\omega)|&\lesssim \omega^{-\frac14}\langle \omega^{-\frac12}r\rangle^{-\frac12}
e^{-\Re[\mu\xi(\mu^{-1/2}r)]} \\
&\lesssim \omega^{-\frac14}\langle \omega^{-\frac12}r\rangle^{-\frac12-\frac{b}{2}}
e^{-\frac{1}{2}r^2}e^{-\varphi(r;\omega,\nu)} \\
&\lesssim \omega^{-\frac14}\langle \omega^{-\frac12}r\rangle^{-\frac12-\frac{b}{2}}
e^{-\frac{1}{2}r^2}
\end{align*}
since $\varphi(\cdot;\omega,\nu)$ is monotonically increasing and $|\varphi(1;\omega,\nu)|\lesssim 1$.
Consequently, with $|v_0(s,\omega)|\lesssim \omega^{-\frac14}$ we infer
\[ |v_0(s,\omega)v_-(r,\omega)|\lesssim \omega^{-\frac12}
\langle \omega^{-\frac12}r\rangle^{-\frac12-\frac{b}{2}}
e^{-\frac{1}{2}r^2} \]
which implies $|\tilde G_\ell(r,s,\omega)|\lesssim \omega^{-\frac12}\langle \omega^{-\frac12}
r\rangle^{-\frac12-\frac{b}{2}}$.

\item Hankel-Hankel: $c\omega^{-\frac12}\leq s\leq r\leq 1$.
From Corollary \ref{cor:rsm} and the Hankel asymptotics we have
\begin{align*} 
|v_0(s,\omega)|&\lesssim |\tilde v_-(s,\omega)|+|\tilde v_+(s,\omega)| \\
&\lesssim \omega^{-\frac14}e^{\Re(\mu^{1/2})s}
\end{align*}
and from the Bessel-Hankel-case we recall the estimate
\[ |v_-(r,\omega)|\lesssim \omega^{-\frac14}e^{\Re(\mu^{1/2})}e^{-\Re(\mu^{1/2})r}
+\omega^{-\frac34}. \]
This yields
\begin{align*} 
|v_0(s,\omega)v_-(r,\omega)|&\lesssim \omega^{-\frac12}e^{\Re(\mu^{1/2})}e^{-\Re(\mu^{1/2})(r-s)} \\
&\quad +\omega^{-1}e^{\Re(\mu^{1/2})s} \\
&\lesssim \omega^{-\frac12}e^{\Re(\mu^{1/2})}
\end{align*}
since $\Re(\mu^{1/2})\geq 0$ and we obtain $|\tilde G_\ell(r,s,\omega)|\lesssim \omega^{-\frac12}$.

\item Hankel-Weber: $c\omega^{-\frac12}\leq s\leq 1\leq r$.
From Hankel-Hankel and Bessel-Weber we infer
\[ |v_0(s,\omega)v_-(r,\omega)|\lesssim \omega^{-\frac12}e^{\Re(\mu^{1/2})}
\langle \omega^{-\frac12}r\rangle^{-\frac12-\frac{b}{2}}
e^{-\frac{1}{2}r^2}
\]
which yields $|\tilde G_\ell(r,s,\omega)|\lesssim \omega^{-\frac12}
\langle \omega^{-\frac12}r\rangle^{-\frac12-\frac{b}{2}}$.

\item Weber-Weber: $1\leq s\leq r$. From Lemma \ref{lem:fs} and Proposition
\ref{prop:fsrbig} we infer
\begin{align*} 
|v_0(s,\omega)|&\lesssim \omega^{-\frac34}e^{\Re(\mu^{1/2})}
\langle \omega^{-\frac12}s\rangle^{-\frac12-\frac{b}{2}}
e^{-\frac{1}{2}s^2}e^{-\varphi(s;\omega,\nu)} \\
&\quad +\omega^{-\frac14}e^{\Re(\mu^{1/2})}
\langle \omega^{-\frac12}s\rangle^{-\frac12+\frac{b}{2}}
e^{\frac{1}{2}s^2}e^{\varphi(s;\omega,\nu)}
\end{align*}
which yields
\begin{align*} 
|v_0(s,\omega)v_-(r,\omega)|&\lesssim \omega^{-\frac12}e^{\Re(\mu^{1/2})}
\langle \omega^{-\frac12}r\rangle^{-\frac12-\frac{b}{2}}\langle \omega^{-\frac12}
s\rangle^{-\frac12+\frac{b}{2}} \\
&\quad \times \big [
\omega^{-\frac12}e^{-\frac{1}{2}r^2-\frac{1}{2}s^2}e^{-\varphi(s;\omega,\nu)}e^{-\varphi(r;\omega,\nu)} \\
&\qquad +e^{-\frac{1}{2}(r^2-s^2)}e^{-[\varphi(r;\omega,\nu)-\varphi(s;\omega,\nu)]}\big ] \\
&\lesssim \omega^{-\frac12}e^{\Re(\mu^{1/2})}
\langle \omega^{-\frac12}r\rangle^{-\frac12-\frac{b}{2}}\langle \omega^{-\frac12}
s\rangle^{-\frac12+\frac{b}{2}} \\
&\quad \times e^{-\frac{1}{2}(r^2-s^2)}
\end{align*}
since $\varphi(\cdot;\omega,\nu)$ is monotonically increasing and $|\varphi(1;\omega,\nu)|\lesssim 1$.
Consequently, we infer 
\[ |\tilde G_\ell(r,s,\omega)|\lesssim \omega^{-\frac12}
\langle \omega^{-\frac12}r\rangle^{-\frac12-\frac{b}{2}}\langle \omega^{-\frac12}
s\rangle^{-\frac12+\frac{b}{2}}. \]
\end{enumerate}
\end{proof}

\subsection{Large angular momenta}
For large angular momenta $\ell\geq\ell_0$ we use the functions
$v_-$ and $\hat v_0$ from Proposition \ref{prop:fsrbig}
and Lemma \ref{lem:fsrsmnlg}, respectively, to construct the operator $\tilde R_\ell(\lambda)$.
For the Wronskian of $v_-$ and $\hat v_0$ we obtain
\begin{equation}
\label{eq:What}
\hat W(\omega):=W(v_-(\cdot,\omega),\hat v_0(\cdot,\omega)) 
=1+O_\C(\omega^{-\frac12}+\nu^{-1})
\end{equation}
where we have used the global representation from Lemma \ref{lem:fsnlg}
and $W(v_-(\cdot,\omega),v_+(\cdot,\omega))=1$ which follows from 
Proposition \ref{prop:fsrbig}.
Eq.~\eqref{eq:What} shows that $\{\hat v_0,v_-\}$ is a fundamental system
for Eq.~\eqref{eq:spechom} provided $\omega$ \emph{and} $\ell$ are sufficiently large
(recall that $\nu=\frac{d}{2}+\ell-1$).
Furthermore, $\hat v_0(r,\omega)$ is recessive as $r\to 0+$ and $v_-(r,\omega)$ is recessive
as $r\to\infty$.
Thus, we set 
\begin{equation}
\label{def:Ghat} \hat G_\ell(r,s,\omega):=\frac{e^{\frac{1}{2}(r^2-s^2)}}{\hat W(\omega)}
\left \{ \begin{array}{ll}
\hat v_0(r,\omega)v_-(s,\omega) & r\leq s \\
v_-(r,\omega)\hat v_0(s,\omega) & r\geq s \end{array} \right .
\end{equation}
and obtain
\[ \tilde R_\ell(\lambda)f(r)=\int_0^\infty \hat G_\ell(r,s,\omega)f(s)ds \]
in the case $\ell \gg 1$.

\subsubsection{Kernel bounds}
Now we show that $\hat G_\ell$ satisfies the same bounds as $\tilde G_\ell$.

\begin{lemma}
\label{lem:estGhat}
The kernel $\hat G_\ell$ defined in \eqref{def:Ghat} satisfies the bound
\[ |\hat G_\ell(r,s,\omega)|\lesssim \omega^{-\frac12}\left \{\begin{array}{l} 
\langle \omega^{-\frac12}r\rangle^{-\frac12+\frac{b}{2}}
\langle \omega^{-\frac12}s\rangle^{-\frac12-\frac{b}{2}}\quad r\leq s \\
\langle \omega^{-\frac12}r\rangle^{-\frac12-\frac{b}{2}}
\langle \omega^{-\frac12}s\rangle^{-\frac12+\frac{b}{2}} \quad r\geq s \end{array}
\right. \]
for all $r,s>0$, $\omega \gg 1$, and $\ell\gg 1$.
\end{lemma}

\begin{proof}
By symmetry it suffices to consider $r\geq s$ and 
we distinguish three cases. As before we set $\alpha=\nu^{-1}\mu^\frac12$
and recall that $\nu=\frac{d}{2}+\ell-1$.
We will use the estimates $|\zeta'(\alpha r)|^{-1}\lesssim 1$
and $|\xi'(\mu^{-\frac12}r)|^{-1}\lesssim \langle \omega^{-\frac12}r\rangle^{-1}$
which follow from Eq.~\eqref{eq:estx}.

\begin{enumerate}
\item Bessel-Bessel: $0<s\leq r\leq 1$. 
From Lemma \ref{lem:fsrsmnlg} we obtain 
\[ |\hat v_0(s,\omega)|\lesssim \omega^{-\frac14}e^{\nu\Re\zeta(\alpha s)} \]
and Lemma \ref{lem:fsnlg} yields
\begin{align*}
|v_-(r,\omega)|&\lesssim |\hat v_0(r,\omega)|+|\hat v_1(r,\omega)| \\
&\lesssim \omega^{-\frac14}e^{\nu\Re \zeta(\alpha r)}+\omega^{-\frac14}e^{-\nu\Re\zeta(\alpha r)}.
\end{align*}
Thus, we infer
\begin{align*} |\hat v_0(s,\omega)v_-(r,\omega)|&\lesssim \omega^{-\frac12}
e^{\nu\Re\zeta(\alpha s)+\nu\Re\zeta(\alpha r)} \\
&\quad +\omega^{-\frac12}e^{-\nu[\Re \zeta(\alpha r)-\Re\zeta(\alpha s)]} \\
&\lesssim \omega^{-\frac12}e^{2\nu\Re\zeta(\alpha)}+\omega^{-\frac12}\lesssim \omega^{-\frac12}
\end{align*}
since $r\mapsto \Re\zeta(\alpha r)$ is monotonically increasing by Lemma \ref{lem:zeta}
and $\zeta(\alpha)=0$.
This yields $|\hat G_\ell(r,s,\omega)|\lesssim \omega^{-\frac12}$.

\item Bessel-Weber: $0<s\leq 1\leq r$.
From Lemma \ref{lem:fsrsmnlg} and Proposition \ref{prop:fsrbig} we infer
\begin{align*} |\hat v_0(s,\omega)v_-(r,\omega)|&\lesssim \omega^{-\frac12}
\langle \omega^{-\frac12}r\rangle^{-\frac12-\frac{b}{2}}e^{-\frac{1}{2}r^2}
e^{\nu \Re\zeta(\alpha s)}e^{-\varphi(r;\omega,\nu)} \\
&\lesssim \omega^{-\frac12}\langle \omega^{-\frac12}r\rangle^{-\frac12-\frac{b}{2}}
e^{-\frac{1}{2}r^2}
\end{align*}
where we have used the fact that $s\mapsto \Re \zeta(\alpha s)$ and $\varphi(\cdot;\omega,\nu)$
are monotonically increasing (Lemmas \ref{lem:zeta}, \ref{lem:xi}) as well as
$\zeta(\alpha)=0$ and $|\varphi(1;\omega,\nu)|\lesssim 1$.
Consequently, we obtain the desired
\[ |\hat G_\ell(r,s,\omega)|\lesssim \omega^{-\frac12}
\langle \omega^{-\frac12}r\rangle^{-\frac12-\frac{b}{2}}. \]
\item Weber-Weber: $1\leq s\leq r$.
Here we use the representation from Lemma \ref{lem:fsnlg}, Proposition \ref{prop:fsrbig}, 
and Lemma \ref{lem:xi}
to obtain
\begin{align*}
|\hat v_0(s,\omega)|&\lesssim |v_-(s,\omega)|+|v_+(s,\omega)| \\
&\lesssim \omega^{-\frac14}\langle \omega^{-\frac12}s\rangle^{-\frac12-\frac{b}{2}}
e^{-\frac{1}{2}s^2}e^{-\varphi(s;\omega,\nu)} \\
&\quad +\omega^{-\frac14}\langle \omega^{-\frac12}s\rangle^{-\frac12+\frac{b}{2}}
e^{\frac{1}{2}s^2}e^{\varphi(s;\omega,\nu)} \\
&\lesssim \omega^{-\frac14}\langle \omega^{-\frac12}s\rangle^{-\frac12+\frac{b}{2}}
e^{\frac{1}{2}s^2}\left [e^{-\varphi(1;\omega,\nu)}+e^{\varphi(s;\omega,\nu)}\right ] 
\end{align*}
and this implies
\begin{align*}
|\hat v_0(s,\omega)v_-(r,\omega)|&\lesssim \omega^{-\frac12}
\langle \omega^{-\frac12}r\rangle^{-\frac12-\frac{b}{2}}
\langle \omega^{-\frac12}s\rangle^{-\frac12+\frac{b}{2}}
e^{-\frac{1}{2}(r^2-s^2)} \\
&\quad \times \left [e^{-2\varphi(1;\omega,\nu)}
+e^{-[\varphi(r;\omega,\nu)-\varphi(s;\omega,\nu)]}
 \right ] \\
 &\lesssim \omega^{-\frac12}
\langle \omega^{-\frac12}r\rangle^{-\frac12-\frac{b}{2}}
\langle \omega^{-\frac12}s\rangle^{-\frac12+\frac{b}{2}}
e^{-\frac{1}{2}(r^2-s^2)}
\end{align*}
since $\varphi(\cdot;\omega,\nu)$ is monotonically increasing and $|\varphi(1;\omega,\nu)|
\lesssim 1$.
\end{enumerate}
\end{proof}

\subsection{Boundedness of the reduced resolvent}

We can now conclude the desired $L^2$-boundedness of $\tilde R_\ell(\lambda)$ for
all angular momenta. This concludes the proof of Theorem \ref{thm:main2}.

\begin{lemma}
\label{lem:L2Rl}
Let $b>0$. Then we have the bound
\[ \|\tilde R_\ell(d+b+\I\omega)\|_{L^2(\R_+)}\leq C \]
for all $\omega\gg 1$ and all $\ell \in \N_0$.
\end{lemma}

\begin{proof}
We choose $\ell_0>0$ so large that Lemma \ref{lem:estGhat} applies for all $\ell\geq \ell_0$.
For $\ell\leq \ell_0$ we define
an auxiliary operator $T_\ell(\omega)$ by
\[ T_\ell(\omega)f(r):=\int_0^\infty \omega^\frac12\tilde G_\ell(\omega^\frac12 r,
\omega^\frac12 s, \omega)f(s)ds. \]
From Lemma \ref{lem:estG} we have the bound
\[ |\omega^\frac12 \tilde G_\ell(\omega^\frac12 r,\omega^\frac12 s,\omega)|
\lesssim \left \{\begin{array}{l} 
\langle r\rangle^{-\frac12+\frac{b}{2}}
\langle s\rangle^{-\frac12-\frac{b}{2}}\quad r\leq s \\
\langle r\rangle^{-\frac12-\frac{b}{2}}
\langle s\rangle^{-\frac12+\frac{b}{2}} \quad r\geq s \end{array}
\right. \] 
and this implies $\|T_\ell(\omega)\|_{L^2(\R_+)}\leq C$ for all $\omega\gg 1$
(see e.g.\ \cite{DonKri13}, Lemma 5.5).
For $a>0$ we write $f_a(r):=f(\frac{r}{a})$ and by scaling we obtain
\begin{align*}
\|\tilde R_\ell(\lambda) f\|_{L^2(\R_+)}
&=\omega^\frac14 \|[\tilde R_\ell(\lambda)f]_{\omega^{-1/2}}\|_{L^2(\R_+)} \\
&=\omega^\frac14 \|T_\ell(\omega)f_{\omega^{-1/2}}\|_{L^2(\R_+)} \\
&\leq C\, \omega^\frac14 \|f_{\omega^{-1/2}}\|_{L^2(\R_+)} \\
&=C \|f\|_{L^2(\R_+)}
\end{align*}
where $\lambda=d+b+\I\omega$.
This proves the statement for all $\ell\leq \ell_0$.
In the case $\ell\geq \ell_0$ we replace $\tilde G_\ell$ by $\hat G_\ell$ and
use Lemma \ref{lem:estGhat}.
\end{proof}

\appendix

\section{Applicability of the abstract theory}
\label{sec:apx}

\noindent In this appendix we discuss the applicability of 
the abstract theory 
to deduce
Theorem \ref{thm:main}.
The most general results related to spectral mapping are developed in
\cite{BreNagPol00}.
To be more precise, the paper \cite{BreNagPol00} deals with the following problem.
Suppose we are given an ``unperturbed semigroup'' $T_0(t)$ on a Banach space $X$
with generator 
$A$ and a ``perturbed semigroup'' $T(t)$ with generator $A+B$ where for our purposes
$B$ may be assumed bounded (the theory in \cite{BreNagPol00} is more general).
Under what assumptions on $T_0$, $B$, and/or $T$ does spectral mapping for $T$ hold?
The paper \cite{BreNagPol00} derives various sufficient criteria based on norm continuity
properties of the remainders in the Dyson-Phillips expansion.
Unfortunately, many of the criteria 
involve the perturbed semigroup $T$ itself or an infinite series of convolutions
of the operators $T_0$ and $B$ which
makes them hard to check. However, there is a set of criteria that involve the unperturbed
semigroup $T_0$ and the perturbing operator $B$ only.
Since there exists an explicit representation of the 
unperturbed Ornstein-Uhlenbeck semigroup $S_0$, 
one might hope to deduce Theorem \ref{thm:main} from the abstract theory.
Let us recall the precise statement.

\begin{theorem}[Brendle-Nagel-Poland \cite{BreNagPol00}]
\label{thm:abstract}
Let $T_0$, $A$, $B$, and $T$ be as above. For $k\in \N_0$ define 
$\tilde T_k: [0,\infty)\to \mc B(X)$ recursively by
\[ \tilde T_{k}(t)f:=\int_0^t \tilde T_{k-1}(t-s)BT_0(s)fds,\qquad \tilde T_0(t):=T_0(t). \]
If there exists a $k\in \N$ such that $\tilde T_k$ is norm continuous on $[0,\infty)$ then
\[ \sigma(T(t))\backslash \{0\}\subset e^{t\sigma(A+B)}\cup \{\lambda: |\lambda|\leq r_{\mathrm{crit}}(T_0(t))\}. \]
Furthermore, if $\tilde T_1$ is norm continuous on $[0,\infty)$ then
\[ \sigma(T(t))\backslash \{0\}=e^{t\sigma(A+B)}\cup \sigma_{\mathrm{crit}}(T_0(t))\backslash \{0\}. \]   
\end{theorem}

\noindent The most accessible criterion is of course the case $k=1$.
In fact, this is also the property that is tested in Example 5.2 in \cite{BreNagPol00}.
In what follows we discuss this criterion for the Ornstein-Uhlenbeck operator $L$.

The free Ornstein-Uhlenbeck semigroup $S_0(t)$ has the explicit representation 
\begin{align} 
\label{eq:OU}
[S_0(t)f](x)&=\frac{1}{[\pi \alpha(t)]^{d/2}}\int_{\R^d} e^{-\frac{|y|^2}{\alpha(t)}}
f(e^{-2t}x-y)dy \nonumber \\
&=:K_t * f(e^{-2t}x)
\end{align}
where $\alpha(t)=1-e^{-4t}$.
Consequently,
$S_0(t)$ consists of a ``heat part'' and a ``dilation part''. 
While the former is very well behaved, the latter is responsible for the difficulties of
the problem at hand. 
Thus, it makes sense to start the discussion by considering the dilation semigroup alone.
To this end, let $[T_0(t)f](x)=f(e^{-2t}x)$ be a dilation semigroup on $L^2(\R^d)$ 
and consider the perturbing operator $Bf(x)=V(x)f(x)$ for some nonzero potential 
$V$, e.g.~$V \in C^\infty_c(\R^d)$.
We are interested in the operator
\[ \tilde T_1(t)f=\int_0^t T_0(t-s)BT_0(s)f ds. \]
Explicitly, we have
\begin{align*} 
[T_0(t-s)BT_0(s)f](x)&=V(e^{-2(t-s)}x)f(e^{-2t}x) \\
 &=[T_0(t-s)V](x)[T_0(t)f](x)
\end{align*}
and this yields
\[ \tilde T_1(t)f=T_0(t)f\int_0^t T_0(t-s)V ds \]
which is certainly \emph{not} norm continuous for all $t\geq 0$ since $t\mapsto T_0(t)$ is
not norm continuous at \emph{any} $t\geq 0$.
Similarly, one sees that the same is true for $\tilde T_k$ for any $k\in \N$.
Consequently, Theorem \ref{thm:abstract} does not apply to perturbed dilation operators, 
not even
if the perturbing potential is as ``nice'' as possible.

The situation for the Ornstein-Uhlenbeck semigroup is similar.
In this case we have
\begin{align*} 
[S_0&(t-s)BS_0(s)f](x) \\
&=\int_{\R^d}K_{t-s}(e^{-2(t-s)}x-y)V(y)[S_0(s)f](y)dy \\
&=\int_{\R^d}K_{t-s}(e^{-2(t-s)}x-y)V(y)\int_{\R^d}K_s(e^{-2s}y-y')f(y')dy' dy \\
&=e^{-2dt}\int_{\R^d}K_{t-s}(e^{-2(t-s)}(x-y))V(e^{-2(t-s)}y) \\
&\qquad \times \int_{\R^d}K_s(e^{-2t}(y-y'))f(e^{-2t}y')dy'
\end{align*}
and thus, for $\tilde S_1(t)f=\int_0^t S_0(t-s)BS_0(s)fds$, we find
\[ \tilde S_1(t)f =e^{-2dt}\int_0^t [T_0(t-s)K_{t-s} * T_0(t-s)V]
[T_0(t)K_s * T_0(t)f]ds. \]
One may now estimate the convolutions using Young's inequality.
However, the presence of the term $T_0(t)f$ leads to 
the same kind of norm discontinuity we encountered in the above discussion of the
dilation operator.
Consequently, we do not see how to apply Theorem \ref{thm:abstract} to the Ornstein-Uhlenbeck
operator $L$. This justifies our approach via explicit construction of the reduced resolvent.

\section{On the spectral inclusion $\sigma(L_0)\subset \sigma(L)$}
\label{sec:specL0}

\noindent By a perturbative argument one can in fact show that the addition of the
potential $V$ does not at all alter the spectrum in the left half-plane
$\{z\in \C: \Re z\leq d\}$.
The point is that the structure of the spectrum of $L_0$ is a consequence of 
the asymptotics of eigenfunctions as $r\to \infty$.
The requirements on the potential, on the other hand, imply 
that the addition of the potential does not change the asymptotic behavior
of eigenfunctions. 
Consequently, one has $\sigma(L_0)\subset \sigma(L)$.
The precise argument is as follows.

\subsection{Proof of Lemma \ref{lem:specL0}}
We use the adjoint operator
\[ L_0^* u(x)=\Delta u(x)+2 x \cdot \nabla u(x)+2d u(x) \]
with domain given in \cite{Met01} and $L^*=L_0^* u+\overline{V}u$.
It suffices to show that
\[ \{z\in \C: \Re z < d\} \subset \sigma_p(L^*). \]
In radial symmetry, the spectral equation $(\lambda-L^*)u=0$ reads
\begin{equation}
\label{eq:u}
 u''(r)+\tfrac{d-1}{r}u'(r)+2 r u'(r)-(\lambda-2d)u(r)=-\overline{V}(r)u(r). 
 \end{equation}
Upon setting $u(r)=r^{-\frac{d-1}{2}}e^{-r^2/2}v(r)$ we find the normal form equation
\begin{equation}
\label{eq:v}
 v''(r)-r^2 v(r)-(\lambda-d)v(r)=[O(r^{-2})-\overline{V}(r)]v(r). 
 \end{equation}
The ``homogeneous'' equation $ v''(r)-r^2 v(r)-(\lambda-d)v(r)=0$ has the fundamental system
\[ U(\tfrac{\lambda-d}{2},\sqrt 2 r),\qquad V(\tfrac{\lambda-d}{2},\sqrt 2 r) \]
where $U$ and $V$ are the standard parabolic cylinder functions (see \cite{NIST}).
We have the bounds
\[ |U(\tfrac{\lambda-d}{2}, \sqrt 2 r)|\simeq e^{-r^2/2}r^{-\frac{\Re\lambda-d+1}{2}},\quad
|V(\tfrac{\lambda-d}{2},\sqrt 2 r)|\simeq e^{r^2/2}r^{\frac{\Re\lambda-d-1}{2}} \]
for $r \gg 1$.
Now one may easily set up a Volterra iteration, see Section \ref{sec:out_Volt},
to treat the right-hand side in Eq.~\eqref{eq:v} perturbatively.
This yields a fundamental system $\{v_0,v_1\}$ of Eq.~\eqref{eq:v} which satisfies the same
bounds as $U(\frac{\lambda-d}{2},\sqrt 2 r)$ and $V(\frac{\lambda-d}{2},\sqrt 2 r)$.
As a consequence, one infers a fundamental system $\{u_0,u_1\}$ of Eq.~\eqref{eq:u}
with
\[ |u_0(r)|\simeq e^{-r^2}r^{-\frac{\Re \lambda}{2}},
\qquad |u_1(r)|\simeq r^{\frac{\Re \lambda}{2}-d} \]
for $r \gg 1$.
This shows that \emph{any} solution of Eq.~\eqref{eq:u} belongs to $L^2_{\mathrm{rad}}(\R^d)$
near infinity provided $\Re \lambda< d$.
Thus, by taking the solution of Eq.~\eqref{eq:u}
which is smooth at $r=0$, we obtain a radial eigenfunction of $L^*$ for any
$\lambda$ with $\Re \lambda<d$.
This proves $\{z\in \C: \Re z<d\}\subset \sigma_p(L^*)$, as desired.

\bibliography{ornuhl}
\bibliographystyle{plain}


\end{document}